\documentclass[oneside,reqno,a4paper,12pt]{amsart}
\usepackage{preamble}
\title{Codistinguished abelian subcategories}
\author{Charley Cummings and David Nkansah}

\begin{document}

\begin{abstract}
We introduce codistinguished abelian subcategories of triangulated categories, which generalise J{\o}rgensen's proper abelian subcategories and are dual to Linckelmann's distinguished abelian subcategories. We show that they come equipped with a non-trivial exact structure induced by the ambient triangulated category, and that every exact structure on an abelian category arises in this way. We also show that fully faithful adjunction triples produce new codistinguished abelian subcategories from old.

\end{abstract}

\maketitle


{
\hypersetup{linkcolor=black}
\tableofcontents
}

\makeatletter
\newcommand{\specialsetup}[1]{%
  \par\addvspace{\medskipamount}%
  \noindent\textbf{#1.}\enspace
  \addcontentsline{toc}{section}{#1}%
  \ignorespaces
}
\makeatother

\section*{Introduction}

Abelian categories and triangulated categories provide natural settings in which to study the representation-theoretic and homological properties of rings. The interaction between these two kinds of categories is often particularly fruitful. This is unsurprising, since many well-known triangulated categories are constructed directly from abelian categories, including homotopy categories, derived categories \cite{verdier}, and singularity categories \cite{buchweitz_singularities}.

Conversely, triangulated categories often contain many interesting abelian subcategories. One celebrated source of such subcategories is provided by hearts of \(t\)-structures, introduced by Beilinson, Bernstein, Deligne, and Gabber in \cite{BBDG}. A heart realises an abelian category as a full subcategory of a triangulated category in such a way that its short exact sequences correspond precisely to distinguished triangles whose three terms lie in the heart. This connection has been exploited in algebraic geometry through Bridgeland stability conditions \cites{bridgeland-stability-2007,bridgeland-stability-2008}, in representation theory through Happel--Reiten--Smal{\o} tilting theory \cites{happel-tilting-1996,woolf-stability-2010}, and in the representation theory of finite groups through the work of Deligne and Lusztig \cite{deligne-representations-1976}.

However, many important triangulated categories arising in representation theory do not admit \(t\)-structures with non-trivial hearts. For example, this occurs for negative cluster categories and stable categories of finite dimensional symmetric algebras \cite{holm-sparseness-2013}*{Thm.~A}. Nevertheless, such triangulated categories may still contain interesting abelian subcategories whose short exact sequences are closely related to the distinguished triangles of the ambient category. Two notions which capture this phenomenon are proper abelian subcategories, introduced by J{\o}rgensen \cite{jorgensen-abelian-2022}, and distinguished abelian subcategories, introduced by Linckelmann \cite{linckelmann-abelian-2024}.

Let \(\clT\) be a triangulated category with suspension functor \(\Sigma\), and let \(\clA\) be a full abelian subcategory of \(\clT\). Following J{\o}rgensen \cite{jorgensen-abelian-2022}, the subcategory \(\clA\) is a \emph{proper abelian subcategory} of \(\clT\) if the following holds:
\[
\exists \text{ an exact sequence }0 \xrightarrow{} a \xrightarrow{f} b \xrightarrow{g} c \xrightarrow{} 0
\text{ in } \clA
\iff
\exists\text{ a triangle }
a \xrightarrow{f} b \xrightarrow{g} c \xrightarrow{} \Sigma a
\text{ in } \clT.
\]
Thus, the short exact sequences in \(\clA\) are precisely those sequences induced by triangles in \(\clT\) whose first three terms lie in \(\clA\).

Proper abelian subcategories directly generalise hearts of \(t\)-structures, but they also occur in triangulated categories which have no non-trivial hearts. J{\o}rgensen showed that they arise naturally in negative cluster categories by taking the extension closure of suitable collections of Hom-orthogonal bricks with sufficiently few negative extensions. Since their introduction, proper abelian subcategories have been developed further in several directions, including Happel--Reiten--Smal{\o} tilting theory \cite{jorgensen-proper-2021}, intermediate categories \cite{kortegaard-intermediate-2025}, and Auslander--Reiten theory \cite{nkansah-nakayama-2026}.

Linckelmann introduced a weaker notion in \cite{linckelmann-abelian-2024}. The subcategory \(\clA\) is a \emph{distinguished abelian subcategory} of \(\clT\) if the following holds:
\[
\exists \text{ an exact sequence }0 \xrightarrow{} a \xrightarrow{f} b \xrightarrow{g} c \xrightarrow{} 0
\text{ in } \clA
\Longrightarrow
\exists\text{ a triangle }
a \xrightarrow{f} b \xrightarrow{g} c \xrightarrow{} \Sigma a
\text{ in } \clT.
\]
Thus, distinguished abelian subcategories generalise proper abelian subcategories by retaining only one of the two implications in their definition. Informally, there are enough triangles in \(\clT\) to account for all short exact sequences in \(\clA\).

Distinguished abelian subcategories arise prominently in stable module categories of finite dimensional selfinjective algebras. In this setting, module categories of suitable quotient algebras can be embedded into a stable module category. A key idea is to construct appropriate functors
\[
\mod{D} \xrightarrow{} \mod{\Lambda}
\]
using the Eilenberg--Watts theorem, and then study their behaviour after postcomposing with the canonical quotient
\[
\mod{\Lambda} \xrightarrow{} \psmod{\Lambda}.
\]

However, neither proper nor distinguished abelian subcategories capture every naturally occurring abelian subcategory of a triangulated category. A basic example is provided by homotopy categories. Consider a module category embedded into a bounded homotopy category by regarding modules as stalk complexes. Only split short exact sequences of modules induce triangles whose first three terms are stalk complexes. Consequently, this embedding is generally neither proper nor distinguished. Nevertheless, every triangle whose first three terms are stalk complexes does induce a short exact sequence of modules. This suggests considering the implication dual to the one appearing in the definition of a distinguished abelian subcategory.

We therefore introduce the following notion. The subcategory \(\clA\) is a \emph{codistinguished abelian subcategory} of \(\clT\) if the following holds:
\[
\exists \text{ an exact sequence }0 \xrightarrow{} a \xrightarrow{f} b \xrightarrow{g} c \xrightarrow{} 0
\text{ in } \clA
\Longleftarrow
\exists\text{ a triangle }
a \xrightarrow{f} b \xrightarrow{g} c \xrightarrow{} \Sigma a
\text{ in } \clT.
\]
Informally, there are enough short exact sequences in \(\clA\) to account for all triangles in \(\clT\) whose first three terms lie in \(\clA\). Codistinguished abelian subcategories are dual to distinguished abelian subcategories. Moreover, an abelian subcategory is proper precisely when it is both distinguished and codistinguished.

This paper is motivated by two basic questions: \textit{why} are codistinguished abelian subcategories interesting, and \textit{where} do they occur? We give some initial answers before describing our main results.

\textbf{Why?}
The first reason is that codistinguished abelian subcategories complete a natural pattern. Proper and distinguished abelian subcategories have both been used successfully to study abelian subcategories of triangulated categories beyond the setting of \(t\)-structures. Since distinguished abelian subcategories retain one of the two implications appearing in the definition of a proper abelian subcategory, it is natural to investigate the dual implication.

Beyond completing this picture, codistinguished abelian subcategories possess useful structural properties. Under natural vanishing assumptions, a codistinguished abelian subcategory \(\clA\) admits desirable approximations; see Proposition~\ref{prop: kernels cover and cokernels envelope}. These approximations allow us to construct adjoints to the inclusion of \(\clA\) into suitable categories of two-term objects in \(\clT\); see Corollary~\ref{C: 2-term adj}. More significantly, extension closed codistinguished abelian subcategories are naturally related to Quillen exact categories.

\textbf{Where?}
Codistinguished abelian subcategories include proper abelian subcategories and hence hearts of \(t\)-structures. They therefore occur wherever proper abelian subcategories occur, for instance in negative cluster categories. They also arise beyond this setting. In particular, they appear naturally in homotopy categories; see Example~\ref{E: homotopy categories}. These examples are typically not proper unless the underlying ring is semisimple.

Codistinguished abelian subcategories can also be constructed from suitable fully faithful embeddings of abelian categories. Suppose that \(\clA\) is a codistinguished abelian subcategory of a triangulated category \(\clT\), and that
\[
\clC \xrightarrow{\iota} \clA
\]
is a fully faithful functor from an abelian category \(\clC\). If \(\iota\) reflects short exact sequences, then its essential image is again a codistinguished abelian subcategory of \(\clT\). We use this observation to study non-exact embeddings whose failure of exactness is controlled. Adjunctions and idempotent subalgebras provide a large supply of such examples.

\specialsetup{Main Results}
We begin by addressing \textit{why}. Our first main result shows that every extension closed codistinguished abelian subcategory carries a natural exact structure whose conflations are precisely the short exact sequences arising from triangles in the ambient triangulated category. This gives an analogue of Dyer's theorem \cite{dyer-exact-2005}; see also \cite{jorgensen-abelian-2022}*{Thm.~3.5(i)}. Dyer's theorem requires the additive subcategory to have vanishing negative first extensions. Our result does not require this vanishing condition; instead, it assumes that the additive subcategory is abelian and codistinguished. Conversely, every exact structure on an abelian category can be realised in this way.

Given an abelian subcategory \(\clA\) of a triangulated category \(\clT\), let \(\clE^\clT_d\) denote the class of short exact sequences
\[
0 \xrightarrow{} a \xrightarrow{} b \xrightarrow{} c \xrightarrow{} 0
\]
in \(\clA\) which induce distinguished triangles
\[
a \xrightarrow{} b \xrightarrow{} c \xrightarrow{} \Sigma a
\]
in \(\clT\).

\begin{ThmIntro}[{Theorem~\ref{thm: exact categories from codistinguished} and Theorem~\ref{T: codistinguished from exact categories}}]
\label{Main Thm: Exact}
The following statements hold:
\begin{enumerate}
    \item Let \(\clA\) be an abelian subcategory of a triangulated category \(\clT\). If \(\clA\) is an extension closed codistinguished abelian subcategory of \(\clT\), then the pair \((\clA,\clE^\clT_d)\) is a Quillen exact category.

    \item Let \(\clA\) be an abelian category, and let \(\clE\) be a class of short exact sequences in \(\clA\) such that \((\clA,\clE)\) is a Quillen exact category. Then there is a triangulated category \(\clD\) which realises \(\clA\) as an extension closed codistinguished abelian subcategory with \(\clE^\clD_d=\clE\).
\end{enumerate}
\end{ThmIntro}

The triangulated category \(\clD\) in Theorem~\ref{Main Thm: Exact}, Part~2, depends on the exact category structure \((\clA,\clE)\). More precisely, we take \(\clD\) to be the bounded derived category of the exact category \((\clA,\clE)\). When \(\clE\) is the class of all short exact sequences in \(\clA\), this is the usual bounded derived category \(\D^b(\clA)\). When \(\clE\) consists only of the split short exact sequences, it is the bounded homotopy category \(\K^b(\clA)\).

Therefore, codistinguished abelian subcategories interpolate between two familiar extremes. At one extreme, the maximal exact structure on \(\clA\) gives \(\D^b(\clA)\), where every short exact sequence induces a triangle. At the other, the split exact structure gives \(\K^b(\clA)\), where only split short exact sequences induce triangles whose terms lie in \(\clA\). More general exact structures describe intermediate behaviour between the derived and homotopy categories. From this perspective, a codistinguished embedding records which short exact sequences of \(\clA\) are visible as triangles in the ambient triangulated category.

Our next result concerns the exact structure from Theorem~\ref{Main Thm: Exact}, Part~1, in the setting of \(2\)-Calabi--Yau triangulated categories. It shows that, provided suitable approximations exist, this exact structure is Frobenius.

\begin{ThmIntro}[{Theorem~\ref{T: Frobenius exact}}]
Let \(k\) be a field. Let \(\clC\) be a Hom-finite idempotent complete \(2\)-Calabi--Yau \(k\)-linear triangulated category. Let \(\clA\) be an extension closed codistinguished abelian subcategory of \(\clC\) such that every object of \(\clA\) admits a \(\Sigma\clA\)-preenvelope. Then the exact category \((\clA,\clE_d^{\clC})\) is a Frobenius exact category.
\end{ThmIntro}

We now address \textit{where}. As explained above, one way to produce codistinguished abelian subcategories is through non-exact fully faithful functors
\[
\clC \xrightarrow{\iota} \clA
\]
between abelian categories which reflect short exact sequences. Adjunctions provide a natural source of such functors. If \(\iota\) admits an exact right adjoint, then \(\iota\) automatically reflects short exact sequences. Since \(\iota\) is then a left adjoint, its possible failure to be exact occurs on the left: a short exact sequence
\[
0 \xrightarrow{} x \xrightarrow{} y \xrightarrow{} z \xrightarrow{} 0
\]
may fail to remain exact at \(\iota x\). Dually, if \(\iota\) admits an exact left adjoint, then it again reflects short exact sequences, but its possible failure to be exact occurs on the right, at \(\iota z\).

Our final main result begins with an adjunction triple \((L,E,R)\) in which \(L\), or equivalently \(R\), is fully faithful. The functors \(L\) and \(R\) provide fully faithful embeddings which reflect short exact sequences, but their possible failures of exactness occur at opposite ends of a short exact sequence. Kuhn studied an \emph{intermediate functor} between \(L\) and \(R\) in \cite{kuhn-generic-1994}*{Sec.~4}. This functor can be viewed as an averaging of \(L\) and \(R\): it is again fully faithful \cite{crawley-boevey-quiver-2017}*{Lem.~2.2} and reflects short exact sequences, but any failure of exactness occurs at the image of the middle term. Our theorem shows that \(L\), \(R\), and the intermediate functor construct new extension closed codistinguished abelian subcategories from old ones.

\begin{ThmIntro}[{Theorem~\ref{T: three codistinguished embeddings} and Corollary~\ref{C: C exact iff C(C) is Serre sub}}]
Let \(\clA\) be an abelian subcategory of a triangulated category \(\clT\), and let \(\clC\) be an additive category. Suppose that we have a fully faithful adjunction triple \((L,E,R)\)
\[
\begin{tikzcd}[row sep=large, column sep=large]
\clA \arrow[r, "E" description]
&
\clC.
\arrow[l, "L"', bend right]
\arrow[l, "R", bend left]
\end{tikzcd}
\]
Then the following statements hold:
\begin{enumerate}
    \item There exists a fully faithful functor
    \(
    \clC \xrightarrow{C} \clA
    \)
    which preserves monomorphisms and epimorphisms and reflects short exact sequences. It is exact if and only if its essential image is a Serre subcategory of \(\clA\).

    \item The essential images \(L\clC\), \(R\clC\), and \(C\clC\) are extension closed in \(\clA\).

    \item If \(\clA\) is a codistinguished abelian subcategory of \(\clT\), then the essential images \(L\clC\), \(R\clC\), and \(C\clC\) are codistinguished abelian subcategories of \(\clT\).
\end{enumerate}
\end{ThmIntro}

Although \(\clC\) is assumed only to be additive, it is automatically abelian in this setting; see Remark~\ref{R: localisation of ab abelian cat is abelian}.

Idempotent subalgebras provide a large source of examples for this construction. Indeed, a fully faithful adjunction triple as in Theorem~\ref{T: three codistinguished embeddings} occurs naturally on the right-hand side of a recollement, and idempotents give a standard class of such recollements. We illustrate this construction in Section~\ref{S: Idempotent subalgebras}. We then study in greater detail the case where a finite dimensional algebra of finite representation type is realised as an idempotent subalgebra of its Auslander algebra.

\specialsetup{Setup}
    Throughout we fix the following setup:
    \begin{itemize}
    \item A \textit{subcategory} of a category will always mean a full subcategory that is closed under isomorphisms.

    \item An \textit{additive subcategory} of an additive category will always mean a full subcategory such that the inclusion functor is additive. In particular, the subcategory is closed under direct sums.
    
    \item We fix a triangulated category \(\clT\) with suspension functor \(\Sigma\).

    \end{itemize}

\section{Codistinguished abelian subcategories}

    Codistinguished abelian subcategories are those abelian subcategories of triangulated categories which reflect exact sequences.
    
    \begin{defn}\label{D: codis}
        An additive subcategory \(\clA\) of \(\clT\) is a \emph{codistinguished abelian subcategory of \(\clT\)} if it is abelian and if, for every triangle
        \[
        a \xrightarrow[]{f} b \xrightarrow[]{g} c \xrightarrow[]{} \Sigma a
        \]
        in \(\clT\) with \(a,b\) and \(c\) in \(\clA\), the sequence
        \[
        0 \xrightarrow[]{} a \xrightarrow[]{f} b \xrightarrow[]{g} c \xrightarrow[]{} 0
        \]
        is a short exact sequence in \(\clA\).
    \end{defn}

    \begin{rem}
    We will sometimes call an abelian subcategory of a triangulated category \emph{codistinguished} if it is a codistinguished abelian subcategory of its ambient triangulated category. Likewise, we will call a fully faithful functor \(\clA \xrightarrow[]{} \clT\) a \emph{codistinguished embedding}, or simply \emph{codistinguished}, if its essential image is a codistinguished abelian subcategory of \(\clT\).
    \end{rem}

    Proper abelian subcategories arise naturally in negative cluster categories, while distinguished abelian subcategories appear naturally in stable module categories of finite dimensional selfinjective algebras. The following example may suggest that codistinguished abelian subcategories arise naturally in homotopy categories.

    \begin{ex}[Homotopy categories]\label{E: homotopy categories}
    Let \(A\) be a ring, and regard \(\Mod{A}\) as a full subcategory of the homotopy category \(\K(\Mod{A})\) consisting of complexes concentrated in degree zero. Then
    \[
    \Mod{A}\subseteq \K(\Mod{A})
    \]
    is a codistinguished abelian subcategory. Moreover, it is proper if and only if \(A\) is semisimple.

    Indeed, suppose
    \[
    M \xrightarrow[]{} N \xrightarrow[]{} L \xrightarrow[]{} M[1]
    \]
    is a triangle in \(\K(\Mod{A})\) with \(M,N\) and \(L\) in \(\Mod{A}\). Since every chain map \(L \xrightarrow[]{} M[1]\) is zero, the connecting morphism must vanish. Hence the triangle splits, and therefore the induced sequence
    \[
    0 \xrightarrow[]{} M \xrightarrow[]{} N \xrightarrow[]{} L \xrightarrow[]{} 0
    \]
    is split exact in \(\Mod{A}\). Thus the embedding is codistinguished.

    On the other hand, a short exact sequence in \(\Mod{A}\) induces such a triangle in \(\K(\Mod{A})\) if and only if it is split, since the corresponding triangle must split. Therefore the embedding is proper if and only if every short exact sequence in \(\Mod{A}\) is split, that is, if and only if \(A\) is semisimple.
    \end{ex}

    \begin{rem}\label{R: bounded homotopy cod}
        The same argument as in Example~\ref{E: homotopy categories} shows that the restricted embedding
        \[
        \mod{A} \subseteq \Kb(\mod{A})
        \]
        is also codistinguished.
    \end{rem}

    The following characterisation is a reinterpretation of the definition of a codistinguished abelian subcategory.
    
    \begin{lem}\label{L:codist_classification}
    Let \(\clA\) be an abelian subcategory of a triangulated category \(\clT\).
    The following are equivalent:
    \begin{enumerate}
        \item[\rm(1)] \(\clA\) is a codistinguished abelian subcategory of \(\clT\).
        \item[\rm(2)] For every triangle
        \[
            a \xrightarrow{f} b \xrightarrow{g} c \xrightarrow{h} \Sigma a
        \]
        in \(\clT\) with \(a\), \(b\) and \(c\) in \(\clA\), the induced natural transformations
        \[
            \restr{\clT(\Sigma a,-)}{\clA}
            \xrightarrow{h^*}
            \restr{\clT(c,-)}{\clA}
            \qquad \text{and} \qquad
            \restr{\clT(-,\Sigma^{-1}c)}{\clA}
            \xrightarrow{(-\Sigma^{-1}h)_*}
            \restr{\clT(-,a)}{\clA}
        \]
        are both zero.
    \end{enumerate}
    \begin{proof}
        \((1\implies 2)\): Let
        \[
            a \xrightarrow{f} b \xrightarrow{g} c \xrightarrow{h} \Sigma a
        \]
        be a triangle in \(\clT\) with \(a\), \(b\) and \(c\) in \(\clA\). Since \(\clA\) is a codistinguished abelian subcategory of \(\clT\), the sequence
        \[
            0 \xrightarrow{} a \xrightarrow{f} b \xrightarrow{g} c \xrightarrow{} 0
        \]
        is short exact in \(\clA\). In particular, \(f\) is a monomorphism in \(\clA\), and \(g\) is an epimorphism in \(\clA\).

        Let \(a'\) be an object in \(\clA\) and let
        \[
            \Sigma a \xrightarrow{v} a'
        \]
        be a morphism in \(\clT\). Since \(h g=0\), as this is a composition of consecutive morphisms in a triangle, we have \((v h) g=0\). The composition \(v h\) lies in \(\clA\) as it is a full subcategory. Therefore, as \(g\) is an epimorphism in \(\clA\), it follows that
        \[
            v h=0.
        \]
        Hence \(h^*\) is zero on \(\clA\).

        Now let \(a'\) be an object in \(\clA\), and let
        \[
            a' \xrightarrow{u} \Sigma^{-1}c
        \]
        be a morphism in \(\clT\). Since \(f (-\Sigma^{-1}h)=0\), as this is a composition of consecutive morphisms in a rotated triangle, we have \(f (-\Sigma^{-1}h) u=0\). The composition  \((-\Sigma^{-1}h) u\) lies in \(\clA\) as it is a full subcategory. Therefore, as \(f\) is a monomorphism in \(\clA\), it follows that
        \[
            (-\Sigma^{-1}h) u=0.
        \]
        Hence \((-\Sigma^{-1}h)_*\) is zero on \(\clA\).

        \((2\implies 1)\): Let
        \[
            a \xrightarrow{f} b \xrightarrow{g} c \xrightarrow{h} \Sigma a
        \]
        be a triangle in \(\clT\) with \(a\), \(b\) and \(c\) in \(\clA\). For every object \(a'\) in \(\clA\), applying \(\clT(-,a')\) and \(\clT(a',-)\) gives exact sequences
        \[
            \clT(\Sigma a,a') \xrightarrow{0} \clT(c,a') \xrightarrow{g^*} \clT(b,a') \xrightarrow{f^*} \clT(a,a')
        \]
        and
        \[
            \clT(a',\Sigma^{-1}c) \xrightarrow{0} \clT(a',a) \xrightarrow{f_*} \clT(a',b) \xrightarrow{g_*} \clT(a',c),
        \]
        where the first morphism in each row is zero by assumption. The first exact sequence shows that \(g\) is the cokernel of \(f\) in \(\clA\) and the second exact sequence shows that \(f\) is the kernel of \(g\) in \(\clA\). Hence
        \[
            0 \xrightarrow{} a \xrightarrow{f} b \xrightarrow{g} c \xrightarrow{} 0
        \]
        is a short exact sequence in \(\clA\). So, \(\clA\) is a codistinguished abelian subcategory of \(\clT\).
    \end{proof}
    \end{lem}

    It immediately follows from Lemma~\ref{L:codist_classification} that any abelian subcategory of a triangulated category without negative first extensions is codistinguished.

    \begin{cor}\label{C: E_1 construction}
        Let \(\clA\) be an abelian subcategory of \(\clT\). If \(\clT(\clA,\Sigma^{-1}\clA)=0\), then \(\clA\) is a codistinguished abelian subcategory of \(\clT\).
    \begin{proof}
    Since
        \(
            \clT(\clA,\Sigma^{-1}\clA)=0,
        \)
        both natural transformations in Lemma~\ref{L:codist_classification} are zero. 
    \end{proof}
    \end{cor}

    \begin{rem}\label{R: abelian subcat of heart is E_1}
        In particular, Corollary~\ref{C: E_1 construction} implies that any abelian subcategory
        of the heart of a \(t\)-structure is codistinguished.
    \end{rem}
    
    \begin{rem}\label{R: converse to E_1}
        The converse of Corollary~\ref{C: E_1 construction} fails in general, as shown in
        Example~\ref{E: negative extensions}.
    \end{rem}

    The following result provides an obstruction to an abelian subcategory of a triangulated category being codistinguished. It shows that, as for hearts of \(t\)-structures, codistinguished abelian subcategories have trivial intersection with their shifts.

    \begin{prop}
    Let \(\clA\) be a codistinguished abelian subcategory of a triangulated category
    \(\clT\). Then \(\Sigma\clA \cap \clA = 0\).
    \begin{proof}
        This proof is the same as the proof of \cite{linckelmann-abelian-2024}*{Cor.\ 6.3}.
        That result is phrased for proper abelian subcategories, but the same reasoning applies
        here without change.

        Let \(a\) be an object in $\clA$ such that $\Sigma a$ is also an object of $\clA$. Then there exists a triangle
        \[
            a \xrightarrow{} 0 \xrightarrow{} \Sigma a \xrightarrow{\id_{\Sigma a}} \Sigma a
        \]
        whose first three terms lie in \(\clA\). Since \(\clA\) is codistinguished, the sequence
        \[
            0 \xrightarrow{} a \xrightarrow{} 0 \xrightarrow{} \Sigma a \xrightarrow{} 0
        \]
        is short exact in \(\clA\). Consequently, it follows that \(a=0\), and so, \(\Sigma\clA \cap \clA=0\).
    \end{proof}
    \end{prop}

    We now turn to codistinguished embeddings arising from Corollary~\ref{C: E_1 construction}. These admit particularly well-behaved approximations.

    Let \(\clY\) be an additive subcategory of an additive category \(\clX\), and let \(X\) be an object of \(\clX\). A strong \emph{\(\clY\)-cover} of \(X\) is a morphism \(Y\xrightarrow{c}X\), with \(Y\) in \(\clY\), such that for each object \(Y'\) in \(\clY\), the induced group homomorphism \(\clX(Y',Y) \xrightarrow{\clX(Y',c)} \clX(Y',X)\) is an isomorphism. The notion of a strong \(\clY\)-envelope is dual.

    \begin{prop}\label{prop: kernels cover and cokernels envelope}
    Let \(\clA\) be an abelian subcategory of \(\clT\) satisfying \(\clT(\clA,\Sigma^{-1}\clA)=0\). Consider the following diagram in \(\clA\):
    \[
    a\xrightarrow[]{f} b \xrightarrow[]{g} c.
    \]
    The following statements are equivalent:
    \begin{enumerate}
        \item The sequence 
        \[
        0 \xrightarrow{} a \xrightarrow[]{f} b \xrightarrow[]{g} c \xrightarrow{} 0
        \]
        is short exact in \(\clA\).
        \item There is a commutative diagram in \(\clT\) of the form
        \begin{equation}\label{diagram: covers and envelopes ses}
            \begin{tikzcd}
                & a \arrow[r, "f"] \arrow[d, "\varepsilon"] & b \arrow[r, "g'"] \arrow[d, equal] & z \arrow[d, "\eta"] \arrow[r] & \Sigma a
                \\
                \Sigma^{-1}c \arrow[r] & x \arrow[r, "f'"] & b \arrow[r, "g"] & c &         
            \end{tikzcd}
        \end{equation}
        with the following properties:
        \begin{itemize}
            \item The top and bottom rows are triangles.
            \item The morphism \(a\xrightarrow[]{\varepsilon}x\) is a strong \(\clA\)-cover.
            \item The morphism \(z \xrightarrow[]{\eta} c\) is a strong \(\clA\)-envelope.
        \end{itemize}
    \end{enumerate}
    \begin{proof}
    We begin with the implication \((1) \implies (2)\).
    
    {\bf Constructing the commutative diagram \eqref{diagram: covers and envelopes ses}.} By the axioms of triangulated categories, we can embed \(f\) and \(g\) into triangles as in \eqref{diagram: covers and envelopes ses}. The morphism \(f'\) is a weak kernel of \(g\), and \(g'\) is a weak cokernel of \(f\). Since \(g f=0\), there exist morphisms \(\varepsilon\) and \(\eta\) making the diagram commute.
        
    {\bf \(\varepsilon\) is a strong \(\clA\)-cover.} Let \(a' \xrightarrow[]{v'} x\) be a morphism with \(a'\) in \(\clA\). The composition \(g(f'v')\) vanishes, since \(g f'\) is a composition of consecutive morphisms in a triangle. As \(f\) is the kernel of \(g\), there exists a unique morphism \(a' \xrightarrow[]{v} a\) rendering the following diagram commutative
    \[
    \begin{tikzcd}
                 & a' \arrow[d, "v'"] \arrow[ldd, "v"', dashed] &   \\
                 & x \arrow[d, "f'"] \arrow[rd, "0"]            &   \\
    a \arrow[r, "f"] & b \arrow[r, "g"]                              & c.
    \end{tikzcd}
    \]
    Now consider the morphism \( a' \xrightarrow[]{\varepsilon v - v'} x\) in \(\clT\). We have that
    \[
    f'(\varepsilon v - v') = f'\varepsilon v - f'v' = f'\varepsilon v - f v = (f'\varepsilon - f)v = 0,
    \]
    where the last equality holds as the left square in \eqref{diagram: covers and envelopes ses} commutes. Since the bottom row in \eqref{diagram: covers and envelopes ses} is a triangle, \(\varepsilon v - v'\) factors through a morphism in \(\clT(a',\Sigma^{-1}c)\), which vanishes by assumption. Therefore, \(\varepsilon v = v'\) and we obtain in \(\clT\) a commutative diagram 
    \[
    \begin{tikzcd}
             & a' \arrow[d, "v'"] \arrow[ld, "v"'] \\
    a \arrow[r, "\varepsilon"] & x.                                
    \end{tikzcd}
    \]
    Thus, \(\varepsilon\) is an \(\clA\)-precover of \(x\).
    
    We show that \(v\) is the unique morphism making the above triangle commute. Suppose \( a' \xrightarrow{\widetilde{v}} a\) satisfies \(\varepsilon \widetilde{v}=0\). Postcomposing with \(f'\) gives
    \[
    0 = f'\varepsilon \widetilde{v} = f \widetilde{v}.
    \]
    Since \(f\) is a monomorphism, it follows that \(\widetilde{v}=0\). Thus, \(\varepsilon\) is a strong \(\clA\)-cover.
    
    {\bf \(\eta\) is a strong \(\clA\)-envelope.} This follows by a dual argument.
    
    We now prove \((2) \implies (1)\).
    
    {\bf \(g\) is the cokernel of \(f\).} Let \(b \xrightarrow[]{w} a'\) be a morphism in \(\clA\) with \(w f=0\). Consider the diagram
    \[
    \begin{tikzcd}
                   & a \arrow[r, "f"] \arrow[d, "\varepsilon"] & b \arrow[r, "g'"] \arrow[d, equal] & z \arrow[d, "\eta"] \arrow[r] \arrow[ldd, "w'", bend left=67] & \Sigma a \\
    \Sigma^{-1}c \arrow[r] & x \arrow[r, "f'"]              & b \arrow[r, "g"] \arrow[d, "w"]                  & c \arrow[ld, "w''"]                                         &          \\
                   &                                 & a'                                                &                                                             &         
    \end{tikzcd}
    \]
    Since \(g'\) is a weak cokernel of \(f\), there exists a morphism \(z \xrightarrow[]{w'} a'\) such that \(w' g'=w\). As \(\eta\) is a strong \(\clA\)-envelope, there exists a unique morphism \(c \xrightarrow[]{w''} a'\) such that \(w''\eta=w'\). By commutativity, \(w'' g = w\).
    
    To prove uniqueness, suppose \(c \xrightarrow{\widetilde{w}''} a'\) also satisfies \(\widetilde{w}''  g=w\). Then \((\widetilde{w}''\eta-w') g'=0\), and so \(\widetilde{w}''\eta-w'\) factors through a morphism in \(\clT(\Sigma a,a')\cong\clT(a,\Sigma^{-1}a')=0\), which vanishes by assumption. Thus, \(\widetilde{w}''\eta=w'\). Since $\eta$ is a strong $\clA$-envelope and \(w''\eta=w'\), it follows that $\widetilde{w}'' = w''$.
    
    {\bf \(f\) is the kernel of \(g\).} This follows by a dual argument.
    \end{proof}
    \end{prop}

    \begin{rem}\label{R: homotopy (co)kernels to (co)kernel functors}
    For additive subcategories \(\clX\) and \(\clZ\) of \(\clT\), we define \(\clX * \clZ\) to be the full subcategory of \(\clT\) consisting of those objects \(t\) in \(\clT\) for which there exists a triangle
    \[
    x \xrightarrow{} t \xrightarrow{} z \xrightarrow{} \Sigma x
    \]
    with \(x\) in \(\clX\) and \(z\) in \(\clZ\).

    Let \(\clA\) be an abelian subcategory of a triangulated category \(\clT\). Firstly, we fix a global choice of (co)kernels for morphisms in \(\clA\), denoted by
    \begin{alignat*}{3}
    f \mapsto \coker(f)
        &\qquad & \text{and} & \qquad
        f \mapsto \ker(f).
    \intertext{Secondly, for each object \(z\) in \(\clA * \Sigma\clA\) and each object \(x\) in \(\Sigma^{-1}\clA * \clA\), fix triangles}
    a_z \xrightarrow[]{f_z} b_z \xrightarrow{} z \xrightarrow{} \Sigma(a_z)
        &\qquad & \text{and} & \qquad
        \Sigma^{-1}(d_x) \xrightarrow[]{} x \xrightarrow{} c_x \xrightarrow[]{g_x} d_x
    \intertext{with \(a_z, b_z, c_x,\) and \(d_x\) in \(\clA\). We define the assignments}
    z \xmapsto{\C_{\clT}} \coker(f_z)
        &\qquad & \text{and} & \qquad
        x \xmapsto{\K_{\clT}} \ker(g_x).
    \end{alignat*}

    Note, by the axioms of triangulated categories, there exist non-unique morphisms \(\eta_z\) and \(\varepsilon_x\) rendering the following diagrams commutative:
    \[
    \begin{tikzcd}[column sep=1.9ex, row sep= small]
                 &                         &                                          &             &            &                            &                                                 & \K_{\clT}(x)=\ker(g_x) \arrow[d] \arrow[ld, "\varepsilon_x"', dashed] &     \\
    a_z \arrow[r, "f_z"] & b_z \arrow[r] \arrow[d] & z \arrow[r] \arrow[ld, "\eta_z", dashed] & \Sigma(a_z) & & \Sigma^{-1}(d_x) \arrow[r] & x \arrow[r] & c_x \arrow[r, "g_x"] & d_x. \\
                 & \C_{\clT}(z)=\coker(f_z)             &                                          &             &            &                            &                                                 &                      &    
    \end{tikzcd}
    \]
    \end{rem}

    Proposition~\ref{prop: kernels cover and cokernels envelope} ensures that the assignments defined in Remark~\ref{R: homotopy (co)kernels to (co)kernel functors} augment to well-defined functors.

    \begin{prop}\label{P: homotopy (co)kernels to (co)kernel functors}
        Let \(\clA\) be an abelian subcategory of \(\clT\) satisfying \(\clT(\clA,\Sigma^{-1}\clA)=0\). The assignments defined in Remark~\ref{R: homotopy (co)kernels to (co)kernel functors} augment to well-defined functors
        \[
        \clA * \Sigma\clA \xrightarrow[]{\C_{\clT}} \clA
         \qquad  \text{and} \qquad
        \Sigma^{-1}\clA * \clA \xrightarrow[]{\K_{\clT}} \clA.
        \]
    \begin{proof}
        We construct the functor \(\clA * \Sigma\clA \xrightarrow[]{\C_{\clT}} \clA\); the other functor is dual. Let
        \[
        z \xrightarrow[]{w} z'
        \]
        be a morphism in \(\clA * \Sigma\clA\). Then we have our fixed choice of triangles
        \[
        a_z \xrightarrow[]{f_z} b_z \xrightarrow{} z \xrightarrow{} \Sigma (a_z)
        \qquad \text{and} \qquad
        a_{z'} \xrightarrow{f_{z'}} b_{z'} \xrightarrow[]{} z' \xrightarrow{} \Sigma (a_{z'})
        \]
        with \(a_z,b_z,a_{z'}\) and \(b_{z'}\) in \(\clA\). By Proposition~\ref{prop: kernels cover and cokernels envelope}, the morphisms \(z \xrightarrow[]{\eta_z} \C_{\clT}(z)\) and \(z' \xrightarrow[]{\eta_{z'}} \C_{\clT}(z')\) that arose by the axioms of triangulated categories in Remark~\ref{R: homotopy (co)kernels to (co)kernel functors} are strong \(\clA\)-envelopes. Since \(\eta_z\) is a strong \(\clA\)-envelope, the composition \(z \xrightarrow[]{w} z' \xrightarrow[]{\eta_{z'}} \C_{\clT}(z')\) induces a unique morphism in \(\clA\)
        \[
        \C_{\clT}(z) \xrightarrow[]{\C_{\clT}(w)} \C_{\clT}(z'),
        \]
        making the following diagram commute
        \[
        \begin{tikzcd}
            z \arrow[d, "w"'] \arrow[r, "\eta_z"] & \C_{\clT}(z) \arrow[ldd, "\C_{\clT}(w)", dashed] \\
            z' \arrow[d, "\eta_{z'}"']            &                                    \\
            \C_{\clT}(z')                                &                                   
        \end{tikzcd}
        \]
        The uniqueness of \(\C_{\clT}(w)\) ensures that the assignment \(w \xmapsto{} \C_{\clT}(w)\) is functorial.
    \end{proof}
    \end{prop}

    \begin{rem}
        The functors constructed in Proposition~\ref{P: homotopy (co)kernels to (co)kernel functors} depend, \emph{a priori}, on three global choices:
        \begin{enumerate}
            \item A choice of (co)kernel for each morphism in \(\clA\).
        
            \item A choice of triangle for each object in \(\clA * \Sigma\clA\) and for each object in \(\Sigma^{-1}\clA*\clA\).
            
            \item A choice of filling morphism between the objects chosen in~(2) and the corresponding (co)kernel chosen in~(1), obtained from the axioms of triangulated categories.
        \end{enumerate}
        The choices in~(1) and~(2) are independent of one another, whereas the choice in~(3) depends on both. However, the following lemma shows that this dependence is harmless, and any such choices give rise to functors that are canonically naturally isomorphic.
    \end{rem}

    \begin{lem}\label{L: you have no choice}
        Consider Remark~\ref{R: homotopy (co)kernels to (co)kernel functors}. Let \(\clA\) be an abelian subcategory of \(\clT\) satisfying \(\clT(\clA,\Sigma^{-1}\clA)=0\).
    \begin{enumerate}
    \item Suppose that we are given two global choices of cokernels for morphisms in \(\clA\), denoted by
        \[
        f \mapsto \coker(f)
        \qquad\text{and}\qquad
        f \mapsto \widetilde{\coker}(f),
        \]
        and that for each object \(z\) in \(\clA * \Sigma\clA\), we are given two commutative diagrams
        \begin{center}
        \begin{tikzcd}[row sep=small, column sep=small]
        a_z \arrow[r, "f_z"] & b_z \arrow[r] \arrow[d] & z \arrow[r] \arrow[ld, "\eta_z"] & \Sigma(a_z)
        & \text{and} &
        a'_z \arrow[r, "f'_z"] & b'_z \arrow[r] \arrow[d] & z \arrow[r] \arrow[ld, "\theta_z"] & \Sigma(a'_z) \\
        & \coker(f_z) & & & & & \widetilde{\coker}(f'_z) & &
        \end{tikzcd}
        \end{center}
        in which the rows are triangles in \(\clT\) and the objects \(a_z,b_z,a'_z\), and \(b'_z\) lie in \(\clA\). Then, by Proposition~\ref{P: homotopy (co)kernels to (co)kernel functors}, the assignment \(z \mapsto \coker(f_z)\) extends to a functor
        \[
        \clA * \Sigma\clA \xrightarrow[]{\C_{\clT}} \clA,
        \]
        and the assignment \(z \mapsto \widetilde{\coker}(f'_z)\) extends to a functor
        \[
        \clA * \Sigma\clA \xrightarrow[]{\widetilde{\C_{\clT}}} \clA.
        \]
        These functors are canonically naturally isomorphic.

    \item Suppose that we are given two global choices of kernels for morphisms in \(\clA\), denoted by
        \[
        g \mapsto \ker(g)
        \qquad\text{and}\qquad
        g \mapsto \widetilde{\ker}(g),
        \]
        and that for each object \(x\) in \(\Sigma^{-1}\clA * \clA\), we are given two commutative diagrams
        \begin{center}
        \begin{tikzcd}[row sep=small, column sep=small]
        &   & \ker(g_x) \arrow[d] \arrow[ld, "\varepsilon_x"'] &     &            &                             &             & \ker(g'_x) \arrow[d] \arrow[ld, "\phi_x"'] &      \\
        \Sigma^{-1}(d_x) \arrow[r] & x \arrow[r] & c_x \arrow[r, "g_x"]                             & d_x & \text{and} & \Sigma^{-1}(d'_x) \arrow[r] & x \arrow[r] & c'_x \arrow[r, "g'_x"]                     & d'_x
        \end{tikzcd}
        \end{center}
        in which the rows are triangles in \(\clT\) and the objects \(c_x,d_x,c'_x\), and \(d'_x\) lie in \(\clA\). Then, by Proposition~\ref{P: homotopy (co)kernels to (co)kernel functors}, the assignment \(x \mapsto \ker(g_x)\) extends to a functor
        \[
        \Sigma^{-1}\clA * \clA \xrightarrow[]{\K_{\clT}} \clA,
        \]
        and the assignment \(x \mapsto \widetilde{\ker}(g'_x)\) extends to a functor
        \[
        \Sigma^{-1}\clA * \clA \xrightarrow[]{\widetilde{\K_{\clT}}} \clA.
        \]
        These functors are canonically naturally isomorphic.
    \end{enumerate}
    \begin{proof}
    \textit{Part 1.} By Proposition~\ref{prop: kernels cover and cokernels envelope}, for each object \(z_1\) in \(\clA * \Sigma\clA\), the morphisms \(\eta_{z_1}\) and \(\theta_{z_1}\) are both strong \(\clA\)-envelopes of the same object. Hence, there exists a unique isomorphism
    \[
    \coker(f_{z_1}) \xrightarrow[]{\alpha_{z_1}} \widetilde{\coker}(f'_{z_1})
    \]
    such that \(\alpha_{z_1}\eta_{z_1}=\theta_{z_1}\). We claim that the collection
    \[
    \{\alpha_{z_1} \mid z_1 \text{ an object in } \clA * \Sigma\clA\}
    \]
    forms the components of a natural isomorphism \(\C_{\clT} \xrightarrow[]{\alpha} \widetilde{\C_{\clT}}\).\footnote{We stress that the components of the natural isomorphism \(\C_{\clT} \xrightarrow[]{} \widetilde{\C_{\clT}}\) are induced from the strong \(\clA\)-envelopes, rather than from the universal property of the cokernels.} Let \(z_1 \xrightarrow[]{w} z_2\) be a morphism in \(\clA * \Sigma\clA\). Then, by definition of the functors \(\C_{\clT}\) and \(\widetilde{\C_{\clT}}\), and the isomorphisms \(\alpha_{z_i}\), we have a diagram
    \[
    \begin{tikzcd}[column sep=small, row sep=small]
        \coker(f_{z_1}) \arrow[ddd, "\C_{\clT}(w)"'] \arrow[rr, "\alpha_{z_1}"] & & \widetilde{\coker}(f'_{z_1}) \arrow[ddd, "\widetilde{\C_{\clT}}(w)"] \\
        & z_1 \arrow[lu, "\eta_{z_1}"] \arrow[ru, "\theta_{z_1}"'] \arrow[d, "w"] & \\
        & z_2 \arrow[ld, "\eta_{z_2}"'] \arrow[rd, "\theta_{z_2}"] & \\
        \coker(f_{z_2}) \arrow[rr, "\alpha_{z_2}"'] & & \widetilde{\coker}(f'_{z_2})
    \end{tikzcd}
    \]
    in which all regions commute, except possibly the outer square with vertices \(\coker(f_{z_1})\), \(\widetilde{\coker}(f'_{z_1})\), \(\coker(f_{z_2})\) and \(\widetilde{\coker}(f'_{z_2})\). We have
    \[
    \widetilde{\C_{\clT}}(w) \alpha_{z_1} \eta_{z_1} = \widetilde{\C_{\clT}}(w) \theta_{z_1} = \theta_{z_2} w = \alpha_{z_2} \eta_{z_2} w = \alpha_{z_2} \C_{\clT}(w) \eta_{z_1},
    \]
    and as \(\eta_{z_1}\) is a strong \(\clA\)-envelope, this implies that
    \[
    \widetilde{\C_{\clT}}(w)\alpha_{z_1} = \alpha_{z_2}\C_{\clT}(w).
    \]
    Thus \(\alpha\) is a natural isomorphism.

    \textit{Part 2.} This is dual to Part~1.
    \end{proof}
    \end{lem}

    The strong \(\clA\)-approximations in Proposition~\ref{prop: kernels cover and cokernels envelope} compare homotopy (co)kernels with honest (co)kernels, and moreover promote the functors of Proposition~\ref{P: homotopy (co)kernels to (co)kernel functors} to adjoints. Before turning to this, note that there are inclusions
    \[
    \clA \subseteq \clA * \Sigma\clA
    \qquad\text{and}\qquad
    \clA \subseteq \Sigma^{-1}\clA * \clA,
    \]
    since each object \(a\) in \(\clA\) fits into the triangle
    \[
    0 \xrightarrow[]{} a \xrightarrow[]{\id_a} a \xrightarrow[]{} 0,
    \]
    and \(0\) lies in both \(\Sigma\clA\) and \(\Sigma^{-1}\clA\).
    
    \begin{cor}\label{C: 2-term adj}
        Consider the strong approximations in Proposition~\ref{prop: kernels cover and cokernels envelope} and the functors in Proposition~\ref{P: homotopy (co)kernels to (co)kernel functors}. Let \(\clA\) be an abelian subcategory of \(\clT\) such that
        \(\clT(\clA,\Sigma^{-1}\clA)=0\). Then the following hold:
        \begin{itemize}
        \item[(1)] The inclusion
        \[
        \clA \xrightarrow{} \clA * \Sigma\clA
        \]
        admits \(\C_{\clT}\) as a left adjoint, with the unit given by \(\eta\).
        
        \item[(2)] The inclusion 
        \[
        \clA \xrightarrow[]{} \Sigma^{-1}\clA * \clA
        \]
        admits \(\K_{\clT}\) as a right adjoint, with the counit given by \(\varepsilon\).
        \end{itemize}
    \begin{proof}
        \emph{Part 1.} We denote the inclusion functor by \(\clA \xrightarrow{\iota} \clA * \Sigma\clA\). Let \(z\) be an object in \(\clA * \Sigma\clA\). The assignment \(z \xmapsto{} \C_{\clT}(z)\) induces a representation of the Hom functor \(\clA * \Sigma\clA(z,\iota(-))\): we have natural isomorphisms
        \[
        \clA*\Sigma\clA(z,\iota(-)) = \clT(z,\iota(-)) \xleftarrow[]{\eta_z^*} \clT(\C_{\clT}(z),\iota(-)) = \clA(\C_{\clT}(z),-),
        \]
        where the first and last equalities hold as \(\clA*\Sigma\clA\) and \(\clA\) are both full subcategories of \(\clT\), and the middle isomorphism holds as \(z\xrightarrow[]{\eta_z}\C_{\clT}(z)\) is a strong \(\clA\)-envelope; see Proposition~\ref{prop: kernels cover and cokernels envelope}. Passing the identity \(\id_{\C_{\clT}(z)}\) through this isomorphism gives the final assertion.

        \emph{Part 2.} This follows similarly to Part~1.
    \end{proof}
    \end{cor}

\section{First examples}

    Recall that a finite dimensional algebra is called an \emph{Auslander algebra} \cite{auslander-representation-2-1974} if its global dimension is at most \(2\) and its dominant dimension is at least \(2\). Equivalently, it is the endomorphism algebra of a basic additive generator of a finite dimensional algebra of finite representation type.

    The following example arises as a special case of the constructions in Section~\ref{S: Codistinguished from FF}; see Remark~\ref{R: Auslander algebra revisited}.

    \begin{ex}[Auslander algebras]\label{E: Auslander algebras2}
        Let \(\Gamma\) be an Auslander algebra. Then the category \(\clA\coloneqq\proj\Gamma\) of finitely generated projective \(\Gamma\)-modules is an abelian category. Indeed, \(\clA\) is equivalent to the module category of a finite dimensional algebra of finite representation type. Viewing \(\mod\Gamma\) as the heart of the canonical \(t\)-structure on \(\Db(\mod\Gamma)\), we show that the embedding
        \[
        \clA \subseteq \mod{\Gamma} \subseteq \Db(\mod{\Gamma})
        \]
        is codistinguished. Moreover, it is proper if and only if \(\clA\) is semisimple, equivalently, if and only if the associated finite dimensional algebra of finite representation type is semisimple.

        Since \(\clA\) is an abelian subcategory of a heart, Remark~\ref{R: abelian subcat of heart is E_1} implies that the embedding is codistinguished. Let \(M\) and \(N\) be \(\Gamma\)-modules. There is an isomorphism
        \[
        \Ext^i_\Gamma(M,N) \cong \Hom_{\Db(\mod\Gamma)}(M,N[i]).
        \]
        If \(M\) lies in \(\clA\), that is, if \(M\) is a projective \(\Gamma\)-module, then the left-hand side vanishes for all \(i\neq 0\).
        
        Now let
        \[
        0 \xrightarrow{} P_1 \xrightarrow[]{} P_2 \xrightarrow[]{} P_3 \xrightarrow{} 0
        \]
        be a short exact sequence in \(\clA\). By the above, this sequence induces a triangle
        \[
        P_1 \xrightarrow[]{} P_2 \xrightarrow[]{} P_3 \xrightarrow[]{} P_1[1]
        \]
        in \(\Db(\mod\Gamma)\) if and only if it is split exact. It follows that the embedding is proper if and only if every short exact sequence in \(\clA\) is split, that is, if and only if \(\clA\) is semisimple.
    \end{ex}

    The reader may notice the similarity between the above argument and that of Example~\ref{E: homotopy categories}. This is because Example~\ref{E: Auslander algebras2} is essentially a disguised version of Example~\ref{E: homotopy categories}.

    \begin{rem}
        Consider Example~\ref{E: Auslander algebras2}. Let \(A\) be a finite dimensional algebra of finite representation type whose Auslander algebra is \(\Gamma\). Then there is an equivalence of abelian categories
        \[
        \proj\Gamma \simeq \mod{A},
        \]
        which induces a triangle equivalence
        \[
        \Kb(\proj\Gamma) \simeq \Kb(\mod{A}).
        \]
        Since every Auslander algebra has finite global dimension, there is also a triangle equivalence
        \[
        \Db(\mod\Gamma) \simeq \Kb(\proj\Gamma).
        \]
        Therefore, the embedding
        \[
        \proj\Gamma \subseteq \Db(\mod\Gamma)
        \]
        from Example~\ref{E: Auslander algebras2} may be reinterpreted as
        \[
        \mod{A} \simeq \proj\Gamma \subseteq \Db(\mod\Gamma) \simeq \Kb(\proj\Gamma) \simeq \Kb(\mod{A}),
        \]
        which sends an \(A\)-module to the corresponding stalk complex concentrated in degree zero. In this way, it recovers the embedding considered in Example~\ref{E: homotopy categories} (see Remark~\ref{R: bounded homotopy cod}).
    \end{rem}

    Co-\(t\)-structures were independently defined by Pauksztello \cite{pauksztello-compact-2008}*{Def.\ 2.4} and Bondarko \cite{bondarko-weight-2010}*{Def.\ 1.1.1}; in the latter reference, co-\(t\)-structures were introduced under the name of weight structures, but we follow the terminology of the former. It was said in \cite{jorgensen-cotstructures-2018} that homotopy categories may be regarded as skewed towards co-\(t\)-structures. As evidence to our claim that codistinguished abelian subcategories arise naturally in homotopy categories, we see that cohearts of the standard co-\(t\)-structure associated to Auslander algebras are codistinguished abelian subcategories.

    \begin{prop}\label{P: coheart codistinguished}
        Let \(\Gamma\) be an Auslander algebra. Then the coheart of the standard co-\(t\)-structure on \(\Kb(\proj\Gamma)\) realises \(\proj\Gamma\) as a codistinguished abelian subcategory of \(\Kb(\proj\Gamma)\).
    \begin{proof}
        This follows from Example~\ref{E: Auslander algebras2}.
    \end{proof}
    \end{prop}

   \begin{ex}[Negative extensions]\label{E: negative extensions}
    Let \(A_n\) denote the path algebra of the quiver
    \[
    1 \leftarrow 2 \leftarrow \cdots \leftarrow n.
    \]
    Consider the bounded derived category \(\Db(\mod{A_3})\). Let
    \[
    \clA=\mod A_1 \oplus \mod A_2,
    \]
    and consider the embedding \(\clA \xrightarrow[]{} \Db(\mod{A_3})\) depicted below. More precisely, we identify \(\clA\) with the additive closure
    \[
    \add\left(
    \substack{2\\1}
    \oplus
    \substack{3}
    \oplus
    \substack{2}\substack{[1]}
    \oplus
    \substack{1}\substack{[2]}
    \right) \subseteq \Db(\mod{A_3}).
    \]
    Thus the image of \(\clA\) is highlighted in the following fragment of the Auslander--Reiten quiver of \(\Db(\mod{A_3})\):
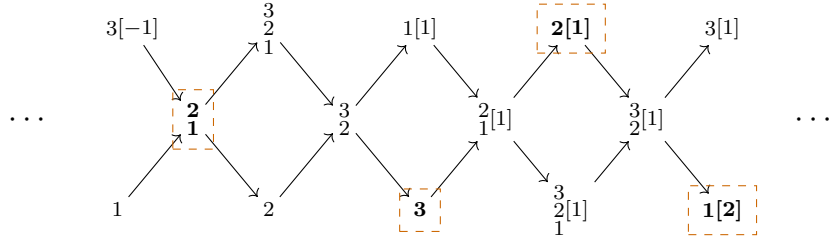
    \[
    \begin{tikzpicture}[>={Computer Modern Rightarrow}, every node/.style={inner sep=2pt}]
        
        \node (dotsL) at (-1.2,0) {\(\cdots\)};

        \node (M1) at (0,-1.2) {\(\substack{1}\)};
        \node (M21) at (1,0) {\(\boldsymbol{\substack{2\\1}}\)};
        \node (M3m1) at (0.2,1.2) {\(\substack{3}\substack{[-1]}\)};

        \node (M321) at (2,1.2) {\(\substack{3\\2\\1}\)};
        \node (M2) at (2,-1.2) {\(\substack{2}\)};
        \node (M32) at (3,0) {\(\substack{3\\2}\)};

        \node (M1p1) at (4,1.2) {\(\substack{1}\substack{[1]}\)};
        \node (M3) at (4,-1.2) {\(\boldsymbol{\substack{3}}\)};
        \node (M21p1) at (5,0) {\(\substack{2\\1}\substack{[1]}\)};

        \node (M2p1) at (6,1.2) {\(\boldsymbol{\substack{2}\substack{[1]}}\)};
        \node (M321p1) at (6,-1.2) {\(\substack{3\\2\\1}\substack{[1]}\)};
        \node (M32p1) at (7,0) {\(\substack{3\\2}\substack{[1]}\)};

        \node (M3p1) at (8,1.2) {\(\substack{3}\substack{[1]}\)};
        \node (M1p2) at (8,-1.2) {\(\boldsymbol{\substack{1}\substack{[2]}}\)};

        \node (dotsR) at (9.2,0) {\(\cdots\)};

        \draw[->] (M1) -- (M21);
        \draw[->] (M3m1) -- (M21);

        \draw[->] (M21) -- (M321);
        \draw[->] (M21) -- (M2);

        \draw[->] (M321) -- (M32);
        \draw[->] (M2) -- (M32);

        \draw[->] (M32) -- (M1p1);
        \draw[->] (M32) -- (M3);

        \draw[->] (M1p1) -- (M21p1);
        \draw[->] (M3) -- (M21p1);

        \draw[->] (M21p1) -- (M2p1);
        \draw[->] (M21p1) -- (M321p1);

        \draw[->] (M2p1) -- (M32p1);
        \draw[->] (M321p1) -- (M32p1);

        \draw[->] (M32p1) -- (M3p1);
        \draw[->] (M32p1) -- (M1p2);

        \node[draw=myorange, dashed, fit=(M21), inner sep=3pt] {};
        \node[draw=myorange, dashed, fit=(M3), inner sep=3pt] {};
        \node[draw=myorange, dashed, fit=(M2p1), inner sep=3pt] {};
        \node[draw=myorange, dashed, fit=(M1p2), inner sep=3pt] {};
    \end{tikzpicture}
    \]

    Every triangle in \(\Db(\mod{A_3})\) of the form
    \[
    L \xrightarrow[]{} M \xrightarrow[]{} N \xrightarrow[]{} L[1],
    \]
    with \(L\), \(M\), and \(N\) in \(\clA\), splits. Hence the embedding is codistinguished. On the other hand, it is not proper. Indeed, the Hom space
    \[
    \Hom_{\Db(\mod{A_3})}\left(\substack{1}\substack{[2]},\substack{3}\substack{[1]}\right)
    \]
    vanishes, whereas there is a non-split short exact sequence
    \[
    0 \xrightarrow[]{} \substack{3} \xrightarrow[]{} \substack{2}\substack{[1]} \xrightarrow[]{} \substack{1}\substack{[2]} \xrightarrow[]{} 0
    \]
    corresponding to the Auslander-Reiten sequence in \(\clA\).

    Finally, note that the irreducible morphism \(\substack{2\\1} \xrightarrow[]{} \substack{2}\) yields a nonzero element of
    \[
    \Hom_{\Db(\mod{A_3})}(\clA,\clA[-1]).
    \]
    Thus this shows that the converse of Corollary~\ref{C: E_1 construction} fails in general; see Remark~\ref{R: converse to E_1}.
    \end{ex}

\begin{rem}
    Linckelmann's example \cite{linckelmann-abelian-2024}*{Ex.\ 9.5} gives an abelian subcategory that is not distinguished, but which turns out to be codistinguished. The ambient triangulated category is \(\psmod\Lambda\), where \(\Lambda\) is a finite dimensional selfinjective algebra, and the embedding is a composition
    \[
        \clB\xrightarrow{}\clA\xrightarrow{}\psmod\Lambda.
    \]
    Here \(\clB\to\clA\) is full, non-exact, and reflects short exact sequences, while \(\clA\to\psmod\Lambda\) is proper by \cite{linckelmann-abelian-2024}*{Cor.\ 6.3}. Hence the composite is codistinguished.
\end{rem}

\section{Codistinguished embeddings and exact structures}

   Consider a codistinguished abelian subcategory of a triangulated category. There are two natural types of short exact sequences: those that induce triangles, and those that do not. This motivates the following definition.

    \begin{defn}\label{D: not dist SES}
        Let \(\clA\) be an abelian subcategory of a triangulated category \(\clT\). A short exact sequence in \(\clA\)
        \[
        0 \xrightarrow[]{} x \xrightarrow[]{f} y \xrightarrow[]{g} z \xrightarrow[]{} 0
        \]
        is \emph{distinguished} if there exists a triangle in \(\clT\) of the form
        \[
        x \xrightarrow[]{f} y \xrightarrow[]{g} z \xrightarrow[]{}\Sigma x.
        \]
        Otherwise, the sequence is \emph{non-distinguished}. Let  \(\clE_d^{\clT}\) denote the collection of all distinguished short exact sequences in \(\clA\).
    \end{defn}

    \begin{rem}
        If \(\clE_d^{\clT}\) contains every short exact sequence in \(\clA\), then \(\clA\) is distinguished in the sense of \cite{linckelmann-abelian-2024}.
    \end{rem}

        In general, discarding the non-distinguished short exact sequences in a codistinguished abelian subcategory does not leave behind any meaningful structure. However, if the codistinguished abelian subcategory is extension closed, then the remaining short exact sequences define an exact structure in the sense of Quillen \cite{quillen-higher-1973}. This gives an analogue of Dyer's theorem \cite{dyer-exact-2005}.

    \begin{thm}\label{thm: exact categories from codistinguished}
        Let \(\clA\) be an extension closed codistinguished abelian subcategory of a triangulated category \(\clT\). Then the pair \((\clA,\clE_d^{\clT})\), see Definition~\ref{D: not dist SES}, is an exact category.
    \begin{proof}
        This statement follows by the proof of \cite{jorgensen-abelian-2022}*{Thm.\ 3.5(i)} (see also \cite{dyer-exact-2005}). One just needs to note that the sequences in \(\clE_d^{\clT}\) are indeed kernel--cokernel pairs by definition, and the codistinguished direction is sufficient to ensure that the arguments used there remain valid.
    \end{proof}
    \end{thm}

    In the next result, we consider the bounded derived category of an exact category. An early treatment of bounded derived categories of exact categories appears in \cite{thomason-higher-1990}*{1.11.6 and App.\ A}, and a more general treatment is given in \cite{neeman-derived-1990}. We also recommend \cites{keller-derived-1996,krause-homological-2022}. The result below shows that codistinguished embeddings arise naturally and in abundance. But first, a little lemma needed for the proof of Theorem~\ref{T: codistinguished from exact categories}.

    \begin{lem}\label{L: compatibility connecting morphisms}
    Let \((\clB,\clE)\) be an exact category and 
    \[
    \xi\colon
    0 \xrightarrow[]{} a \xrightarrow[]{f} b \xrightarrow[]{g} c \xrightarrow[]{} 0
    \]
    be a conflation in \((\clB,\clE)\). Let
    \[
    a \xrightarrow[]{f} b \xrightarrow[]{g} c \xrightarrow[]{h_\xi} a[1]
    \]
    be the triangle in \(\bDer(\clB,\clE)\) associated to \(\xi\) by
    \cite{krause-homological-2022}*{Lem.\ 4.1.12}. Then, under the natural isomorphism
    \[
    \Ext_{(\clB,\clE)}(c,a)
    \xrightarrow[]{\Phi}
    \Hom_{\bDer(\clB,\clE)}(c,a[1])
    \]
    of \cite{krause-homological-2022}*{Prop.\ 4.2.11}, one has
    \[
    \Phi([\xi])=h_\xi.
    \]
    \begin{proof}
      Let
        \[
        Q\colon \Kb(\clB)\xrightarrow[]{}\bDer(\clB,\clE)
        \]
        denote the localisation functor. We regard objects of \(\clB\) as stalk
        complexes concentrated in degree \(0\). 
        Denote the cone of the morphism $f$ by $C_f$. Then in \(\Kb(\clB)\) there is a triangle
        \[
        a\xrightarrow[]{f} b \xrightarrow[]{} C_f \xrightarrow[]{\pi} a[1].
        \]
        The proof of \cite{krause-homological-2022}*{Lem.\ 4.1.12} uses the
        quasi-isomorphism $ C_f\xrightarrow[]{q}c$
        induced by \(b\xrightarrow[]{g} c\). 
        Hence the connecting morphism of the triangle associated to \(\xi\) is
        \[
        h_\xi
        =
        c\xrightarrow[]{Q(q)^{-1}} C_f \xrightarrow[]{Q(\pi)} a[1].
        \]
        
        On the other hand, \cite{krause-homological-2022}*{Prop.\ 4.2.11} associates
        to \(\xi\) the complex
        \[
        E_\xi\colon
        \cdots \xrightarrow[]{} 0 \xrightarrow[]{} b
        \xrightarrow[]{g} c \xrightarrow[]{} 0 \xrightarrow[]{} \cdots,
        \]
        with \(b\) in degree \(-1\) and \(c\) in degree \(0\). Then the identity morphism on $c$ induces a morphism of complexes $c\xrightarrow[]{\xi_{\id_c}}E_\xi$ and the morphism $a \xrightarrow[]{f} b$ induces a morphism of complexes $a[1]\xrightarrow[]{\xi_f}E_\xi$.
        The morphism \(\xi_f\) is a quasi-isomorphism and by construction, we have
        \[
        \Phi([\xi])
        =
        c \xrightarrow[]{Q(\xi_{\id_c})} E_{\xi}
        \xrightarrow[]{Q(\xi_f)^{-1}} a[1].
        \]
        We claim that the diagram
        \[
        \begin{tikzcd}
        C_f \arrow[d, "q"'] \arrow[r, "\pi"] & a[1] \arrow[d, "\xi_f"] \\
        c \arrow[r, "\xi_{\id_c}"']                & E_\xi
        \end{tikzcd}
        \]
        commutes in \(\Kb(\clB)\). Indeed, consider the difference
        \[
        \xi_f\pi-\xi_{\id_c}q\colon C_f\xrightarrow[]{}E_\xi,
        \]
        which is represented by
        \[
        \begin{tikzcd}[column sep=2.8em]
        C_f\colon\quad
        \cdots \arrow[r] &
        0 \arrow[r] &
        a \arrow[r, "f"] \arrow[d, "-f"'] &
        b \arrow[r] \arrow[d, "-g"'] &
        0 \arrow[r] &
        \cdots \\
        E_\xi\colon\quad
        \cdots \arrow[r] &
        0 \arrow[r] &
        b \arrow[r, "g"] &
        c \arrow[r] &
        0 \arrow[r] &
        \cdots .
        \end{tikzcd}
        \]
        This morphism is null-homotopic. Hence
        \[
        Q(\xi_f)Q(\pi)=Q(\xi_{\id_c})Q(q)
        \]
        in \(\bDer(\clB,\clE)\). Since both \(q\) and \(\xi_f\) are
        quasi-isomorphisms, it follows that
        \[
        h_\xi
        =
        Q(\pi)Q(q)^{-1}
        =
        Q(\xi_f)^{-1}Q(\xi_{\id_c})
        =
        \Phi([\xi]). \qedhere
        \]
    \end{proof}
    \end{lem}

    \begin{thm}\label{T: codistinguished from exact categories}
        Let \(\clA\) be an abelian category, and suppose that a class of short exact sequences \(\clE\) endows \(\clA\) with the structure of an exact category \((\clA,\clE)\). Then the embedding
        \[
            \clA \xrightarrow[]{} \bDer(\clA,\clE),
        \]
        sending an object to the stalk complex concentrated in degree \(0\), realises \(\clA\) as an extension closed codistinguished abelian subcategory. Moreover, if there exists a short exact sequence in \(\clA\) that does not lie in \(\clE\), then the embedding is not proper.
    \begin{proof}
        Let
        \[
        a\xrightarrow[]{f} b \xrightarrow[]{g} c \xrightarrow[]{h} a[1]
        \]
        be a triangle in which \(a\), \(b\), and \(c\) are objects of \(\clA\). By
        \cite{krause-homological-2022}*{Prop.\ 4.2.11}, there is a natural isomorphism
        \[
        \Ext_{(\clA,\clE)}(c,a)
        \xrightarrow[]{\Phi}
        \Hom_{\bDer(\clA,\clE)}(c,a[1]).
        \]
        Choose a conflation
        \[
        \xi\colon
        0 \xrightarrow[]{} a \xrightarrow[]{f'} b' \xrightarrow[]{g'} c \xrightarrow[]{} 0
        \]
        representing \(\Phi^{-1}(h)\). By \cite{krause-homological-2022}*{Lem.\ 4.1.12}, this conflation induces a triangle
        \[
        a\xrightarrow[]{f'} b' \xrightarrow[]{g'} c \xrightarrow[]{h_\xi} a[1].
        \]
        By Lemma~\ref{L: compatibility connecting morphisms}, we have
        \[
        h_\xi=\Phi([\xi])=h.
        \]
        By the axioms of triangulated categories, there is an isomorphism \(b\xrightarrow[]{}b'\) rendering the following diagram commutative
        \[
        \begin{tikzcd}
            a \arrow[r, "f"] \arrow[d, equal] & b \arrow[r, "g"] \arrow[d, dashed] & c \arrow[r, "h"] \arrow[d, equal] & \Sigma{a} \arrow[d, equal] \\
            a \arrow[r, "f'"]                 & b' \arrow[r, "g'"]                 & c \arrow[r, "h_\xi"]                 & \Sigma{a}.
        \end{tikzcd}
        \]
        Therefore the sequence
        \[
        0 \xrightarrow[]{} a \xrightarrow[]{f} b \xrightarrow[]{g} c \xrightarrow[]{} 0
        \]
        is isomorphic to the conflation \(\xi\). Since conflations are closed under isomorphism, it is itself a conflation. The final statement follows immediately from the bijection \(\Phi\) and Lemma~\ref{L: compatibility connecting morphisms}.
    \end{proof}
    \end{thm}

    \begin{rem}\label{R: Correspondence between exact cat and ext closed cod}
    Let \(\clA\) be an abelian category. Theorem~\ref{thm: exact categories from codistinguished} and Theorem~\ref{T: codistinguished from exact categories} yield a non-bijective correspondence between the following:
    \begin{itemize}
        \item[1.] Exact structures \(\clE\) on \(\clA\).
        \item[2.] Triangulated categories \(\clT\) equipped with an extension closed codistinguished embedding of \(\clA\).
    \end{itemize}

    The assignments from \(1\) to \(2\) and from \(2\) to \(1\) are given by
    \begin{alignat*}{3}
        (\clA,\clE)
            &\qquad& \xmapsto{\varphi} &\qquad& {}&
            \clA \xrightarrow[]{} \bDer(\clA,\clE)
        \\
        \shortintertext{and, conversely,}
        \clA \xrightarrow[]{} \clT
            &\qquad& \xmapsto{\psi} &\qquad& {}&
            (\clA,\clE_d^{\clT}).
    \end{alignat*}
    These assignments are well-defined by Theorem~\ref{T: codistinguished from exact categories} and Theorem~\ref{thm: exact categories from codistinguished}, respectively. By construction, \(\psi\varphi\) is the identity. Although the composite \(\varphi\psi\) need not be the identity, if \(\clT(\clA,\Sigma^{<0}\clA)=0\), then it induces a realisation functor by \cite{keller-sous-1987}*{3.2 Thm.\ on p.\ 228}; see also \cite{letz-realization-2025}*{Thm.\ 1.1}.
\end{rem}
    
    We have seen that an extension closed codistinguished abelian subcategory of a triangulated category \(\clT\) gives rise to a natural exact category whose short exact sequences are precisely those induced by triangles in \(\clT\). In the next result, we show that in a \(2\)-Calabi--Yau setting this exact category is in fact Frobenius.

    Let \(\clY\) and \(\clZ\) be additive subcategories of an additive category \(\clX\), and let \(Y\) be an object of \(\clY\). A \emph{\(\clZ\)-preenvelope} of \(Y\) is a morphism \(Y\xrightarrow{e}Z\), with \(Z\) in \(\clZ\), such that for each object \(Z'\) in \(\clZ\), the induced group homomorphism \(\clX(Z,Z') \xrightarrow{\clX(e,Z')} \clX(Y,Z')\) is surjective.

    \begin{thm}\label{T: Frobenius exact}
        Let \(k\) be a field. Let \(\clC\) be a Hom-finite idempotent complete \(2\)-Calabi--Yau \(k\)-linear triangulated category. Let \(\clA\) be an extension closed codistinguished abelian subcategory of \(\clC\) such that every object of \(\clA\) admits a \(\Sigma\clA\)-preenvelope. Then the exact category \((\clA,\clE_d^{\clC})\), from Theorem~\ref{thm: exact categories from codistinguished}, is a Frobenius exact category.
    \begin{proof}
        By Theorem~\ref{thm: exact categories from codistinguished}, \((\clA,\clE_d^{\clC})\) is an exact category. By \cite{jorgensen-abelian-2022}*{Prop.\ 3.6}, the existence of \(\Sigma\clA\)-preenvelopes ensures that the exact category has enough projective objects. Since \(\clC\) is \(2\)-Calabi--Yau, we have natural isomorphisms
        \[
        \clC(x,\Sigma(-)) \cong \D^2\clC(x,\Sigma(-)) \cong \D\clC(-,\Sigma x),
        \]
        for each object \(x\) in \(\clC\). Hence \(\clC(x,\Sigma a)=0\) for all \(a\) in \(\clA\) if and only if \(\clC(a,\Sigma x)=0\) for all \(a\) in \(\clA\). It follows that the projective objects coincide with the injective objects. Therefore \((\clA,\clE_d^{\clC})\) is a Frobenius exact category.
    \end{proof}
    \end{thm}

\section{Codistinguished embeddings from fully faithful adjunction triples}\label{S: Codistinguished from FF}

    In this section, we construct new codistinguished abelian subcategories from existing ones. We begin with an adjunction triple between additive categories \((L,E,R)\)
    \[
    \begin{tikzcd}[row sep=large, column sep = large]
    \clA \arrow[r, "E" description] & \clC, \arrow[l, "L"', bend right] \arrow[l, "R", bend left]
    \end{tikzcd}
    \]
    where both \(L\) and \(R\) are fully faithful. In the situation that \(\clA\) is a codistinguished abelian subcategory of a triangulated category \(\clT\), we show that such an adjunction triple will produce three possibly distinct realisations of \(\clC\) as a codistinguished abelian subcategory of \(\clT\). In particular, we do not assume \textit{a priori} that \(\clC\) is an abelian category; this will follow automatically.
    
    Two of these realisations follow directly from the properties of the functors \(L\) and \(R\). The third is more subtle: it is obtained from a canonical functor
    \[
    \clC\xrightarrow[]{C}\clA
    \]
    which lies between \(L\) and \(R\). Following \cite{kuhn-generic-1994}*{Sec.\ 4}, we call this the \emph{intermediate extension functor} associated to the adjunction triple \((L,E,R)\).
    
\subsection{Fully faithful adjunction triples}
    
    We introduce some basic operations internal to the 2-category of small categories. For an accessible introduction, we recommend \cite{riehl-elements-2022}*{App.\ B}.
    
    Let \(\clX\xrightarrow[]{F,G}\clY\) be two functors. A natural transformation \(F\xrightarrow[]{\alpha}G\) between them may be depicted as
     \[
        \xymatrix@R=2cm@C=1.5cm{
        &\clX \rtwocell<5>^F_G{\alpha}
        &\clY.
        }
    \]
    
    A pair of natural transformations \(F\xrightarrow[]{\alpha}G\xrightarrow[]{\beta}H\) admits a \emph{(vertical) composite} \(F\xrightarrow{\beta\alpha}H\):
    \vspace{-4ex}
    \[
    \xymatrix@R=2cm @C=1.5cm{
    \clX \ruppertwocell<9>^F{\alpha}
    \ar[r]|{\: G\:}
    \rlowertwocell<-9>_H{\beta}
    & \clY
    &= 
    &\clX \rtwocell<5>^F_H{\hspace{0.56em}\beta\alpha}
    &\clY,
    }
    \]
    \vspace{-3ex}\\
    whose component at an object \(x\) in \(\clX\) is the morphism \(Fx\xrightarrow[]{\beta_x\alpha_x}Hx\) in \(\clY\). 
    
    Similarly, a pair of natural transformations 
    \[
    \xymatrix@R=2cm @C=1.5cm{
    \clX \rtwocell<5>^F_G{\alpha} &\clY \rtwocell<5>^I_J{\gamma}&\clZ
    }
    \]
    admits a \emph{horizontal composite}
    \[
    \xymatrix@R=2cm @C1.5cm{
    \clX \rtwocell<5>^{IF}_{JG}{\hspace{0.8em}\gamma*\alpha}&\clZ,
    }
    \]
    whose component at an object \(x\) in \(\clX\) is the diagonal morphism of the naturality square in \(\clZ\):
    \[
    \begin{tikzcd}[row sep=1.2cm, column sep=1.2cm]
        IFx \arrow[d, "I(\alpha_x)"] \arrow[r, "\gamma_{Fx}"] & JFx \arrow[d, "J(\alpha_x)"] \\
        IGx \arrow[r, "\gamma_{Gx}"]                          & JGx.             
    \end{tikzcd}
    \]
    
    Two useful special cases arise by taking \(\alpha=\id_F\) or \(\gamma=\id_I\). In these cases,  we call the horizontal composites \(I\alpha\coloneqq\id_I*\alpha\) and \(\gamma F\coloneqq\gamma*\id_F\), \emph{whiskered composites}, and depict them as
    \[
    \xymatrix@R=2cm @C=1.5cm{
    \clX \rtwocell<5>^F_G{\alpha} &\clY \ar[r]^{I}&\clZ,
    }
    \]
    \[
    \xymatrix@R=2cm @C=1.5cm{
    \clX \ar[r]^{F} &\clY \rtwocell<5>^I_J{\gamma}&\clZ.
    }
    \]

    Vertical and horizontal composition interact coherently: a diagram of natural transformations
    \vspace{-3ex}
    \[
    \xymatrix@R=2cm @C1.5cm{
    \clX \ruppertwocell<9>^F{\alpha}
    \ar[r]|{\: G\:}
    \rlowertwocell<-9>_H{\beta}
    & \clY \ruppertwocell<9>^I{\gamma}
    \ar[r]|{\: J\:}
    \rlowertwocell<-9>_K{\delta}
    & \clZ
    }
    \]
    \vspace{-3ex}\\
    satisfies the \emph{interchange law} \((\delta\gamma)*(\beta\alpha) = (\gamma*\alpha)(\delta*\beta)\).

    As our aim is to work with adjunction triples whose leftmost and rightmost functors are fully faithful, we now analyse how fully faithful functors interact with adjunctions. The following lemma characterises objects in the essential image of a fully faithful adjoint.

    \begin{lem}\label{L: EImg of FF adjoint}
        Let \(\clX\) and \(\clY\) be categories and consider an adjunction pair \((F,G)\)
    \[
    \begin{tikzcd}
    \clX \arrow[r, "F", shift left] & \clY. \arrow[l, "G", shift left]
    \end{tikzcd}
    \]
    Suppose that \(F\) is fully faithful. Let \(\id_{\clX}\xrightarrow[]{\eta}GF\) and \(FG \xrightarrow[]{\varepsilon} \id_{\clY}\) denote the unit and the counit of the adjunction, respectively. Then an object \(y\) in \(\clY\) is in the essential image of \(F\) if and only if \(\varepsilon_y\) is an isomorphism.
    \begin{proof}
    The if implication \((\impliedby)\) is clear. Conversely, suppose that \(y\) is in the essential image of \(F\). Then there exists an object \(x\) in \(\clX\) and an isomorphism \(Fx\xrightarrow{\alpha} y\). Since \(F\) is fully faithful, the unit \(\id_{\clX}\xrightarrow[]{\eta} GF\) is a natural isomorphism; see \cite{maclane-categories-1998}*{IV.3, Thm.\ 1}. Hence \(Fx\xrightarrow[]{F\eta_x} FGFx\) is an isomorphism. By the triangle identity, we have \(\varepsilon_{Fx}\circ F\eta_x=\id_{Fx}\). Since \(F\eta_x\) is an isomorphism, it follows that \(\varepsilon_{Fx}\) is an isomorphism.

    Now consider the naturality square for the counit applied to the isomorphism \(\alpha\):
    \[
    \begin{tikzcd}
        FGFx \arrow[r, "FG\alpha"] \arrow[d, "\varepsilon_{Fx}"'] 
            & FGy \arrow[d, "\varepsilon_y"] \\
            Fx \arrow[r, "\alpha"'] 
            & y .
    \end{tikzcd}
    \]
    The morphisms \(\alpha\), \(FG\alpha\), and \(\varepsilon_{Fx}\) are isomorphisms. Therefore, by commutativity, \(\varepsilon_y\) is also an isomorphism.
    \end{proof}
    \end{lem}

    \begin{lem}\label{L: Norm_fully faith adjunction triple}
    Let \(\clX\) and \(\clY\) be categories and consider an adjunction triple \((L,E,R)\)
    \[
    \begin{tikzcd}[row sep=large, column sep = large]
    \clX \arrow[r, "E" description] & \clY. \arrow[l, "L"', bend right] \arrow[l, "R", bend left]
    \end{tikzcd}
    \]
    Then \(L\) is fully faithful if and only if \(R\) is fully faithful. 
    Let the unit and counit of \((L,E)\) be denoted by
    \begin{equation*}
    \id_{\clY} \xrightarrow{\eta^L} EL
    \;\; \text{ and } \;\;
    LE \xrightarrow{\varepsilon^L} \id_{\clX}
    \end{equation*}
    respectively, and let the unit and counit of \((E,R)\) be denoted by
    \begin{equation*}
    \id_{\clX} \xrightarrow{\eta^R} RE
    \;\; \text{ and } \;\;
    ER \xrightarrow{\varepsilon^R} \id_{\clY}
    \end{equation*}
    respectively. If $L$ is fully faithful, then the following statements hold:
    \begin{itemize}
        \item[1.] There exists a natural transformation \(L\xrightarrow[]{\:\mathscr{n}}R\) whose component \(\mathscr{n}_y\) at an object \(y\) in \(\clY\) corresponds to the identity morphism \(\id_y\) under the diagonal natural isomorphism \(\clX(Ly,Ry)\xrightarrow[]{\Phi_{y}}\clY(y,y)\) of the commutative square
        \begin{equation}\label{eqn: diag nat isom def N}
        \begin{tikzcd}
            {\clX(Ly,Ry)} \arrow[d, "(\operatorname{adj}^E_R)^{-1}", swap] \arrow[r, "\operatorname{adj}^L_E"] & {\clY(y,ERy)} \arrow[d, "(\varepsilon^R_y)_*"] \\
            {\clY(ELy,y)} \arrow[r, "(\eta^L_y)^*", swap]                               & {\clY(y,y)}.
        \end{tikzcd}
        \end{equation}
        Here \(\operatorname{adj}^L_E\) and \(\operatorname{adj}^E_R\) denote the adjunction isomorphisms associated to the adjunction pairs \((L,E)\) and \((E,R)\), respectively.
        \item[2.] The whiskered composite \(EL\xrightarrow[]{E\mathscr{n}}ER\) coincides with the composite \[EL\xrightarrow[]{(\eta^L)^{-1}}\id_{\clY}\xrightarrow[]{(\varepsilon^R)^{-1}}ER.\]
        In particular, \(E\mathscr{n}\) is a natural isomorphism.
        \item[3.] The whiskered composite \(LE\xrightarrow[]{\mathscr{n} E}RE\) coincides with the composite \[LE\xrightarrow[]{\varepsilon^L}\id_{\clX}\xrightarrow[]{\eta^R}RE.\]
    \end{itemize}
\begin{proof}

    It follows from \cite{borceux-handbook-1-1994}*{Prop.\ 3.4.2} that $L$ is fully faithful if and only if $R$ is fully faithful.

    \textit{Part 1.} Since \(R\) is fully faithful, the counit \(\varepsilon^R\) is a natural isomorphism. Define \(\mathscr{n}\) to be the composite
    \[
    \mathscr{n}\colon L \xrightarrow[]{L((\varepsilon^R)^{-1})} LER \xrightarrow[]{\varepsilon^LR} R.
    \]

    For an object \(y\) in \(\clY\), we compute the image of \(\mathscr{n}_y\) under the diagonal natural isomorphism \(\clX(Ly,Ry)\xrightarrow[]{\Phi_y}\clY(y,y)\) in diagram~\eqref{eqn: diag nat isom def N}. We have
    \begin{align*}
        \Phi_y(\mathscr{n}_y)&=\varepsilon^R_y \circ E\mathscr{n}_y \circ \eta^L_y\\
                &=\varepsilon^R_y \circ E(\varepsilon^L_{Ry})\circ E(L(\varepsilon^R_y)^{-1}) \circ \eta^L_y\\
                &=\varepsilon^R_y \circ E(\varepsilon^L_{Ry}) \circ \eta^L_{ERy}\circ (\varepsilon^R_y)^{-1}\\
                &=\varepsilon^R_y \circ (\varepsilon^R_y)^{-1}\\
                &=\id_y,
    \end{align*}
    where the first equality is the definition of \(\Phi_y\), the second equality holds by functoriality of \(E\), the third equality holds by naturality of \(\eta^L\) (induced by the morphism \((\varepsilon^R_y)^{-1}\)), and the fourth equality follows from the triangle identity for the adjunction pair \((L,E)\). This shows that \(\mathscr{n}_y\) corresponds to \(\id_y\) under \(\Phi_y\), as claimed.

    \textit{Part 2.} From the first and last equality in the above calculation, we have
    \[
    \varepsilon^R_y \circ E\mathscr{n}_y \circ \eta^L_y=\id_y,
    \]
    for all objects \(y\) in \(\clY\). Rearranging yields \(E\mathscr{n}_y=(\varepsilon^R_y)^{-1} \circ (\eta^L_y)^{-1}\) which proves the claim.

    \textit{Part 3.} We verify that \(\eta^R_a\circ \varepsilon^L_a\) corresponds to the identity morphism \(\id_{Ea}\) under the diagonal natural isomorphism \(\clX(LEa,REa) \xrightarrow[]{\Phi_{Ea}} \clY(Ea,Ea)\). Indeed,
    \begin{align*}
        \Phi_{Ea}(\eta^R_a\circ \varepsilon^L_a)&=\varepsilon^R_{Ea} \circ E(\eta^R_a\circ \varepsilon^L_a) \circ \eta^L_{Ea}\\
                &=\varepsilon^R_{Ea} \circ E(\eta^R_a)\circ E(\varepsilon^L_a) \circ \eta^L_{Ea}\\
                &=\id_{Ea},
    \end{align*}
    where the first equality is the definition of \(\Phi_{Ea}\), the second equality holds by functoriality of \(E\), and the third equality follows from the triangle identity for the adjunction pairs \((L,E)\) and \((E,R)\). Thus \(\mathscr{n}_{Ea}=\eta^R_a\circ \varepsilon^L_a\).
\end{proof}
\end{lem}

We call an adjunction triple \((L,E,R)\) \emph{fully faithful} if either \(L\) or \(R\) is fully faithful. The natural transformation \(L\xrightarrow[]{\mathscr{n}} R\) constructed in Lemma~\ref{L: Norm_fully faith adjunction triple} is called the \emph{norm} of the fully faithful adjunction triple \((L,E,R)\) in \cite{franjou-comparison-2004}. We now use the norm to construct another functor $\clC \to \clA$ known as the \emph{intermediate extension functor}.

    Let \(\clY\) be a category. Its \emph{category of arrows} \(\Arr\clY\) has as objects the morphisms of \(\clY\) and as morphisms commutative squares in \(\clY\). Using the category of arrows, a natural transformation
    \[
    \xymatrix@R=2cm@C=1.5cm{
    \clX \rtwocell<5>^F_G{\alpha} & \clY
    }
    \]
    can be reinterpreted as a functor
    \[
    \clX\xrightarrow[]{\alpha} \Arr\clY,
    \]
    which sends an object \(x\) in \(\clX\) to the component \(\alpha_x\), and a morphism to the corresponding naturality square in \(\clY\). In an additive setting, the induced functor is additive.
    
    \begin{lem}\label{L: norm is additive}
    Let \(\clX\) and \(\clY\) be additive categories and let \(\clX\xrightarrow[]{F,G}\clY\) be two additive functors. If \(F\xrightarrow[]{\alpha}G\) is a natural transformation, then the induced functor
    \[
    \clX \xrightarrow[]{\alpha}\Arr\clY
    \]
    is additive.
    \begin{proof}
    The functor \(\alpha\) sends the zero object \(0\) of \(\clX\) to the morphism \(F0\xrightarrow[]{\alpha_0}G0\) in \(\clY\). Since \(F\) and \(G\) are additive, we have \(F0\cong G0\cong 0\), and hence \(\alpha_0\) is the zero object in \(\Arr\clY\).
    
    Let \(x\) and \(y\) be objects in \(\clX\). It now suffices to show that \(\alpha_{x\oplus y}\cong \alpha_x\oplus\alpha_y\) in \(\Arr\clY\). Choose a biproduct diagram
    \begin{equation}\label{Eqn: biprod F,G(x+y)}
    \begin{tikzcd}
        x \arrow[r, "\iota_x", shift left] & x\oplus y \arrow[l, "\pi_x", shift left] \arrow[r, "\pi_y", shift left] & y, \arrow[l, "\iota_y", shift left]
    \end{tikzcd}
    \end{equation}
    exhibiting \(x\oplus y\) as the direct sum of \(x\) and \(y\). Since \(F\) and \(G\) are additive, the diagrams
    \[
    \begin{tikzcd}
        Fx \arrow[r, "F(\iota_x)", shift left] & F(x\oplus y) \arrow[l, "F(\pi_x)", shift left] \arrow[r, "F(\pi_y)", shift left] & Fy \arrow[l, "F(\iota_y)", shift left] &\text{and} & Gx \arrow[r, "G(\iota_x)", shift left] & G(x\oplus y) \arrow[l, "G(\pi_x)", shift left] \arrow[r, "G(\pi_y)", shift left] & Gy \arrow[l, "G(\iota_y)", shift left]
        \end{tikzcd}
    \]
    exhibit \(F(x\oplus y)\) as the direct sum of \(Fx\) and \(Fy\), and \(G(x\oplus y)\) as the direct sum of \(Gx\) and \(Gy\), respectively.
    
    The biproduct diagram~\eqref{Eqn: biprod F,G(x+y)} induces the following eight naturality squares:
    \[
    \begin{tikzcd}[row sep=1.2cm]
        Fx \arrow[d, "F(\iota_x)"', shift right] \arrow[r, "\alpha_x"]                                                   & Gx \arrow[d, "G(\iota_x)"', shift right]                                             \\
        F(x\oplus y) \arrow[u, "F(\pi_x)"', shift right] \arrow[d, "F(\pi_y)"', shift right] \arrow[r, "{\alpha_{x\oplus y}}"] & G(x\oplus y) \arrow[u, "G(\pi_x)"', shift right] \arrow[d, "G(\pi_y)"', shift right] \\
        Fy \arrow[u, "F(\iota_y)"', shift right] \arrow[r, "\alpha_y"]                                                   & Gy. \arrow[u, "G(\iota_y)"', shift right]
    \end{tikzcd}
    \]
    It follows from the commutativity of these naturality squares that the above diagram is a biproduct diagram in \(\Arr\clY\), which exhibits \(\alpha_{x\oplus y}\) as the direct sum of \(\alpha_x\) and \(\alpha_y\).
    \end{proof}
    \end{lem}
    
    Let \(\clA\) be an abelian category. The assignment sending a morphism in \(\clA\) to its image augments to an additive functor
    \[
    \Arr\clA \xrightarrow[]{\Img} \clA
    \]
    (see \cite{maclane-categories-1998}*{VIII.1, Lem.\ 1 and VIII.3, Prop.\ 1}). In general, this functor is neither full nor faithful. However, the upcoming proposition shows that if we restrict to morphisms arising from the components of the norm of a fully faithful adjunction triple, the resulting restricted image functor is fully faithful. Before proceeding, a useful note...

    In what follows, our aim is to construct new codistinguished abelian subcategories from old ones by realising the new category as a full subcategory of the original ambient abelian category via a fully faithful adjunction triple. The next remark shows that, in this setting, it suffices to assume that the smaller category is additive---abelianness is then automatic.

\begin{rem}\label{R: localisation of ab abelian cat is abelian}
    Let \(\clA\) be an abelian category and let \(\clC\) be an additive category. Suppose we have a functor \(\clA \xrightarrow[]{E}\clC\) that preserves finite limits. If \(E\) admits either a fully faithful left adjoint or a fully faithful right adjoint, then \(\clC\) is automatically an abelian category (see \cite{borceux-handbook-2-1994}*{Prop.\ 1.4.3}).
\end{rem}

\begin{lem}\label{L: N Fact, F(pi), F(i) iso}
    Let \(\clA\) be an abelian category and \(\clC\) be an additive category. Let \((L,E,R)\) be a fully faithful adjunction triple
    \[
    \begin{tikzcd}[row sep=large, column sep = large]
    \clA \arrow[r, "E" description] & \clC, \arrow[l, "L"', bend right] \arrow[l, "R", bend left]
    \end{tikzcd}
    \]
    and let \(L\xrightarrow[]{\mathscr{n}}R\) be the natural transformation constructed in Lemma~\ref{L: Norm_fully faith adjunction triple}. Let $C$ denote the functor
    \[
    C \coloneqq \clC \xrightarrow[]{\mathscr{n}} \Arr\clA \xrightarrow[]{\Img} \clA.
    \]
    Then \(C\) induces a canonical factorisation
    \vspace{-2ex}
    \[
    \xymatrix@C+1.5em{
    \clC 
    \ruppertwocell<8>^L{\pi}
    \ar[r]|{C} 
    \rlowertwocell<-8>_R{\iota}
    &
    \clA
    &= 
    &\clC \rtwocell<5>^L_R{\hspace{0.46em}\mathscr{n}}
    &\clA,
    }
    \]
    \vspace{-3ex}\\
    of the norm \(\mathscr{n}\), with the following properties:
    \begin{itemize}
        \item[1.] The natural transformation \(L\xrightarrow[]{\pi}C\) is componentwise an epimorphism and the whiskered composite \(EL\xrightarrow[]{E\pi}EC\) is a natural isomorphism.
        \item[2.] The natural transformation \(C\xrightarrow[]{\iota}R\) is componentwise a monomorphism, and the whiskered composite \(EC\xrightarrow[]{E\iota}ER\) is a natural isomorphism.
    \end{itemize}
\begin{proof}
    We first note that by Remark~\ref{R: localisation of ab abelian cat is abelian}, the category \(\clC\) is an abelian category. We now describe the action of \(C\) on objects and morphisms.
        \begin{itemize}
            \item \textit{Objects:} For each object \(x\) in \(\clC\), assign the image \(Cx\), which comes equipped with the canonical factorisation
            \[
            Lx\xrightarrow[]{\pi_x} Cx \xrightarrow[]{\iota_x} Rx
            \]
            of \(\mathscr{n}_x\). Notice that \(\pi_x\) is an epimorphism and \(\iota_x\) is a monomorphism.
            \item \textit{Morphisms:} For each morphism \(x\xrightarrow[]{f}y\) in \(\clC\), assign the unique morphism
            \[
            Cx \xrightarrow[]{Cf} Cy
            \]
            making the following diagram commute:
            \begin{equation}\label{eqn: fact_of_norm}
            \begin{tikzcd}
            Lx \arrow[r, "\pi_x"] \arrow[d, "Lf"] &   Cx \arrow[r, "\iota_x"] \arrow[d, "Cf", dashed] & Rx \arrow[d, "Rf"] \\
            Ly \arrow[r, "\pi_y"]                 & Cy \arrow[r, "\iota_y"]                                   & Ry.
            \end{tikzcd}
            \end{equation}
        \end{itemize}
    
    By construction, \(\mathscr{n}\) factors as the vertical composite \(\iota\pi\). The commutative squares in diagram~\eqref{eqn: fact_of_norm} show that \(\{\pi_x\mid x\text{ in } \clC\}\) and \(\{\iota_x\mid x\text{ in } \clC\}\) are the components of natural transformations \(\pi\) and \(\iota\), respectively.

    \textit{Part~1.} Clearly, \(\pi\) is a componentwise epimorphism. As \(E\) is exact, the whiskered composite \(EL\xrightarrow[]{E\pi}EC\) is componentwise an epimorphism. To see that it is also componentwise a monomorphism, observe that by the interchange law, the whiskered composite \(E\mathscr{n}\) equals the vertical composite \(E\iota \, E\pi\). By Lemma~\ref{L: Norm_fully faith adjunction triple}, \(E\mathscr{n}\) is a natural isomorphism, which implies that \(E\pi\) is also a componentwise monomorphism, hence a natural isomorphism.

    \textit{Part~2.} This follows similarly to Part~1.
\end{proof}
\end{lem}

\begin{prop}\label{P: lift properties}
    Let \(\clA\) be an abelian category and \(\clC\) be an additive category. Let \((L,E,R)\) be a fully faithful adjunction triple
    \[
    \begin{tikzcd}[row sep=large, column sep = large]
    \clA \arrow[r, "E" description] & \clC, \arrow[l, "L"', bend right] \arrow[l, "R", bend left]
    \end{tikzcd}
    \]
    and let \(\id_\clC \xrightarrow[]{\eta^L}EL\) denote the unit of the adjunction \((L,E)\). Let \(L\xrightarrow[]{\mathscr{n}}R\) be the natural transformation constructed in Lemma~\ref{L: Norm_fully faith adjunction triple}. Let $C$ denote the functor
    \[
    C \coloneqq \clC \xrightarrow[]{\mathscr{n}} \Arr\clA \xrightarrow[]{\Img} \clA.
    \]
    Then \(C\) satisfies the following properties:
    \begin{itemize}
        \item[1.] The functor \(C\) is additive.
        \item[2.] The functor \(C\) is fully faithful.
        \item[3.] The composition \(\id_\clC \xrightarrow{\eta^L}EL \xrightarrow[]{E\pi}EC\), where \(L\xrightarrow[]{\pi}C\) is the natural transformation constructed in Lemma~\ref{L: N Fact, F(pi), F(i) iso}, is a natural isomorphism.
        \item[4.] If \(f\) is an epimorphism in \(\clC\), then \(Cf\) is an epimorphism in \(\clA\).
        \item[5.] If \(f\) is a monomorphism in \(\clC\), then \(Cf\) is a monomorphism in \(\clA\).
    \end{itemize}
    \begin{proof}
    We first note that by Remark~\ref{R: localisation of ab abelian cat is abelian}, the category \(\clC\) is an abelian category.
    
        \textit{Part 1.} This follows as \(C\) is a composition of additive functors (see Lemma~\ref{L: norm is additive}). Alternatively, additivity follows from Part~2 of this proposition since fully faithful functors between additive categories are automatically additive (see for example, \cite{kashiwara-categories-2006}*{Cor.\ 8.2.16}).
        
        \textit{Part 2.} See \cite{crawley-boevey-quiver-2017}*{Lem.\ 2.2}.

        \textit{Part~3.} This follows as the unit \(\eta^L\) of the adjunction pair \((L,E)\) is a natural isomorphism since \(L\) is fully faithful, and \(E\pi\) is a natural isomorphism by Lemma~\ref{L: N Fact, F(pi), F(i) iso}, Part~1.

        \textit{Parts 4 and 5.} See \cite{kuhn-generic-1994}*{Lem.\ 2.1, Part 4}.
    \end{proof}
\end{prop}

\subsection{Three codistinguished embeddings}

We now show that each of $L$, $R$ and the intermediate extension functor $C$ realises $\clC$ as a codistinguished abelian subcategory. 

\begin{setup}\label{Set: FF adj 3's}
    We fix an abelian category \(\clA\), an additive category \(\clC\) and a fully faithful adjunction triple \((L,E,R)\)
    \[
    \begin{tikzcd}[row sep=large, column sep = large]
    \clA \arrow[r, "E" description] & \clC. \arrow[l, "L"', bend right] \arrow[l, "R", bend left]
    \end{tikzcd}
    \]
    Let
    \[
    \id_{\clC} \xrightarrow{\eta^L} EL
    \quad\text{and}\quad
    LE \xrightarrow{\varepsilon^L} \id_{\clA}
    \]
    denote the unit and counit of the adjunction \((L,E)\) and let
    \[
    \id_{\clA} \xrightarrow{\eta^R} RE
    \quad\text{and}\quad
    ER \xrightarrow{\varepsilon^R} \id_{\clC}
    \]
    denote the unit and counit of the adjunction \((E,R)\). By Remark~\ref{R: localisation of ab abelian cat is abelian}, the category \(\clC\) is an abelian category. Let
    \[
    L\xrightarrow[]{\mathscr{n}}R
    \]
    be the norm natural transformation constructed in Lemma~\ref{L: Norm_fully faith adjunction triple}. Following \cite{kuhn-generic-1994}*{Sec.\ 4}, we define the fully faithful \emph{intermediate extension functor} associated to the adjunction triple \((L,E,R)\) to be the composite
    \[
    C\colon
    \clC \xrightarrow[]{\mathscr{n}} \Arr\clA \xrightarrow[]{\Img} \clA.
    \]
\end{setup}

\begin{prop}\label{P: LB,RB, LiftB extension closed in A}
    Consider Setup~\ref{Set: FF adj 3's}. Each of the essential images $L\clC$, $C\clC$ and $R\clC$ is extension closed in $\clA$.
\begin{proof}
    \textit{Part 1.} Let \(x\) and \(z\) be objects in \(\clC\) and suppose
    \[
    0\xrightarrow[]{} Lx \xrightarrow[]{} a \xrightarrow[]{} Lz \xrightarrow[]{} 0
    \]
    is a short exact sequence in \(\clA\). We must show that \(a\) lies in \(L\clC\).
    
    Using the exactness of \(E\), the right exactness of \(L\), and the naturality of the counit \(\varepsilon^L\), we obtain a commutative diagram in \(\clA\)
    \[
    \begin{tikzcd}
            & LELx \arrow[d, "\varepsilon^L_{Lx}"] \arrow[r] & LEa \arrow[d, "\varepsilon^L_a"] \arrow[r] & LELz \arrow[d, "\varepsilon^L_{Lz}"] \arrow[r] & 0 \\
    0 \arrow[r] & Lx \arrow[r]                                 & a \arrow[r]                              & Lz \arrow[r]                                 & 0
    \end{tikzcd}
    \]
    whose rows are exact. Since \(L\) is fully faithful, the counit \(\varepsilon^L_t\) is an isomorphism if and only if \(t\) lies in the essential image of \(L\) (see Lemma~\ref{L: EImg of FF adjoint}). Hence \(\varepsilon^L_{Lx}\) and \(\varepsilon^L_{Lz}\) are isomorphisms, and the Snake Lemma implies that \(\varepsilon^L_a\) is also an isomorphism. Thus Lemma~\ref{L: EImg of FF adjoint} ensures \(a\) is in the essential image of \(L\), as required.

    \textit{Part 2.} This follows similarly to Part 1.

    \textit{Part 3.} Let \(x\) and \(z\) be objects in \(\clC\) and suppose
    \begin{equation}\label{Eqn: Extension closure in Lift}
    0\xrightarrow[]{} Cx \xrightarrow[]{f} a \xrightarrow[]{g} Cz \xrightarrow[]{} 0
    \end{equation}
    is a short exact sequence in \(\clA\). We must show that \(a\) lies in \(C\clC\).
    
    By Lemma~\ref{L: Norm_fully faith adjunction triple}, Part~3, the component
    \(\mathscr{n}_{Ea}\colon LEa\to REa\) of the norm factors as
    \[
    LEa \xrightarrow[]{\varepsilon^L_a} a \xrightarrow[]{\eta^R_a} REa.
    \]
    We show that \(\varepsilon^L_a\) is an epimorphism and that
    \(\eta^R_a\) is a monomorphism. Hence this factorisation is the canonical
    image factorisation of \(\mathscr{n}_{Ea}\), and so \(a\) is exhibited as
    \(CEa\).

    Applying the exact functor \(E\) to \eqref{Eqn: Extension closure in Lift}, we obtain an exact sequence in \(\clC\)
    \begin{equation}\label{Eqn: 1st Extension closure in Lift}
    0\xrightarrow[]{} E(Cx) \xrightarrow[]{Ef} Ea \xrightarrow[]{Eg} E(Cz) \xrightarrow[]{} 0.
   \end{equation}
    By Proposition~\ref{P: lift properties}, Part 3, there is a natural isomorphism
    \[
    \alpha\coloneqq\id_\clC \xrightarrow{\eta^L}EL \xrightarrow[]{E\pi}EC.
    \]
    Using \(\alpha\), the exactness of \eqref{Eqn: 1st Extension closure in Lift} implies that
    \begin{equation}\label{Eqn: 2nd Extension closure in Lift}
    0\xrightarrow[]{} x \xrightarrow[]{E(f)\alpha_x} Ea \xrightarrow[]{\alpha_z^{-1}Eg} z \xrightarrow[]{} 0
    \end{equation}
    is exact in \(\clC\). Applying the right exact functor \(L\) and left exact functor \(R\) to \eqref{Eqn: 2nd Extension closure in Lift}, we consider the following diagram
    \begin{equation}\label{Eqn: Scary diagram}
    \begin{tikzcd}[column sep=large]
        & Lx \arrow[r, "L(E(f)\alpha_x)"] \arrow[d, "\pi_x"] & LEa \arrow[r, "L(\alpha^{-1}_zEg)"] \arrow[d, "\varepsilon^L_a"] & Lz \arrow[r] \arrow[d, "\pi_z"]               & 0 \\
        0 \arrow[r] & Cx \arrow[r, "f"] \arrow[d, "\iota_x"] & a \arrow[r, "g"] \arrow[d, "\eta^R_a"]                           & Cz \arrow[r] \arrow[d, "\iota_z"] & 0 \\
        0 \arrow[r] & Rx \arrow[r, "R(E(f)\alpha_x)"]                    & REa \arrow[r, "R(\alpha^{-1}_zEg)"]                              & Rz,                                        &  
    \end{tikzcd}
    \end{equation}
    whose rows are exact. Recall from Lemma~\ref{L: N Fact, F(pi), F(i) iso} that the components of \(L\xrightarrow[]{\pi}C\xrightarrow[]{\iota}R\) provide canonical image factorisations of the components of the norm \(L\xrightarrow[]{\mathscr{n}} R\).
    
    We claim that diagram~\eqref{Eqn: Scary diagram} commutes. Transposing diagram~\eqref{Eqn: Scary diagram} through the adjunctions \((L,E)\) and \((E,R)\), this is equivalent to the commutativity of the diagram
    \begin{equation}\label{Eqn: Adj of Scary diagram}
    \begin{tikzcd}[column sep=large, row sep=large]
        & x \arrow[r, "E(f)\alpha_x"] \arrow[d, "E(\pi_x)\eta^L_x"]                & Ea \arrow[r, "\alpha^{-1}_zEg"] \arrow[d, equal] & z \arrow[r] \arrow[d, "E(\pi_z)\eta^L_z"]                          & 0 \\
        0 \arrow[r] & E(Cx) \arrow[r, "Ef"] \arrow[d, "\varepsilon^R_xE(\iota_x)"] & Ea \arrow[r, "Eg"] \arrow[d, equal]              & E(Cz) \arrow[r] \arrow[d, "\varepsilon^R_zE(\iota_z)"] & 0 \\
        0 \arrow[r] & x \arrow[r, "E(f)\alpha_x"]                                              & Ea \arrow[r, "\alpha^{-1}_zEg"]                 & z .                                                                &  
    \end{tikzcd}
    \end{equation}
    The top two squares commute by definition of \(\alpha\). The bottom-right square commutes for similar reasons as the bottom-left. For the bottom left square, commutativity will occur provided that the composition
    \begin{equation}\label{eqn: be the identity}
    E(Cx)\xrightarrow[]{\varepsilon^R_xE(\iota_x)}x\xrightarrow[]{E(\pi_x)\eta^L_x}E(Cx)
    \end{equation}
    is the identity. Since \(Lx \xrightarrow[]{\pi_x}Cx\xrightarrow[]{\iota_x}Rx\) is the canonical image factorisation of \(\mathscr{n}_x\), transposing through the adjunctions shows that the reverse composition
    \[
    x \xrightarrow[]{E(\pi_x)\eta^L_x} E(Cx) \xrightarrow[]{\varepsilon^R_xE(\iota_x)} x
    \]
    is the identity (see Lemma~\ref{L: Norm_fully faith adjunction triple}, Part 1). This implies three things:
    \begin{itemize}
        \item The morphism \(E(\pi_x)\eta^L_x\) is invertible, since \(E\) is exact and \(\pi_x\) is an epimorphism.
        \item The morphism \(\varepsilon^R_xE(\iota_x)\) is invertible, since \(E\) is exact and \(\iota_x\) is a monomorphism.
        \item The previous composition~\eqref{eqn: be the identity} is an idempotent, which by the previous two points, is invertible.
    \end{itemize}
    But, the only invertible idempotent is the identity, so we are done.
    
    Finally, applying the Snake Lemma twice to diagram~\eqref{Eqn: Scary diagram}, we conclude that \(\varepsilon^L_a\) is an epimorphism and \(\eta^R_a\) is a monomorphism. Thus the factorisation \(\mathscr{n}_{Ea}=\eta^R_a\circ \varepsilon^L_a\) is indeed the canonical image factorisation, and therefore \(a\) lies in \(C\clC\).
\end{proof}
\end{prop}
    
A functor $F \colon \clC \to \clA$ between abelian categories \emph{reflects short exact sequences in $\clA$}, if for all short exact sequences
    \[
    0\xrightarrow[]{} Fx \xrightarrow[]{Ff} Fy\xrightarrow[]{Fg} Fz \xrightarrow[]{} 0
    \]
    in $\clA$, the sequence
    \[
    0\xrightarrow[]{} x \xrightarrow[]{f} y\xrightarrow[]{g} z \xrightarrow[]{} 0
    \]
    is a short exact sequence in $\clC$.
    The functors \(L,R\) and \(C\) in Setup~\ref{Set: FF adj 3's} are fully faithful functors that, while not necessarily exact, always reflect short exact sequences.

\begin{prop} \label{P: reflect ses}
    Consider Setup~\ref{Set: FF adj 3's}. Each of the functors $L$, $R$ and $C$ reflects short exact sequences. Moreover, the following hold:
    \begin{itemize}
        \item[1.] $L$ is exact if and only if it preserves monomorphisms.
        \item[2.] $R$ is exact if and only if it preserves epimorphisms.
        \item[3.] $C$ is exact if and only if for every short exact sequence
        \[
        0 \to x \xrightarrow[]{f} y \xrightarrow[]{g} z \to 0
        \]
        in $\clC$ either of the following holds:
        \begin{itemize}
            \item[(a)] For each object $a$ in $\clA$, the sequence of abelian groups
            \[
            \clA(a,Cx) \xrightarrow[]{(Cf)_*} \clA(a,Cy) \xrightarrow[]{(Cg)_*} \clA(a,Cz)
            \]
            is exact.
            \item[(b)] For each object $a$ in $\clA$, the sequence of abelian groups
            \[
            \clA(Cz,a) \xrightarrow[]{(Cg)_*} \clA(Cy,a) \xrightarrow[]{(Cf)_*} \clA(Cx,a)
            \]
            is exact.
        \end{itemize}
    \end{itemize}
\end{prop}
\begin{proof}
    Suppose there exists a short exact sequence
    \[
    0 \xrightarrow[]{} Lx \xrightarrow[]{Lf} Ly \xrightarrow[]{Lg} Lz \xrightarrow[]{} 0.
    \]
    Since \(E\) is an exact functor and the unit \(\id_\clC\xrightarrow[]{\eta^L}EL\) of the adjunction pair \((L,E)\) is a natural isomorphism, the diagram
    \[
    0 \xrightarrow[]{} x \xrightarrow[]{f} y \xrightarrow[]{g} z \xrightarrow[]{} 0
    \]
    is a short exact sequence in \(\clC\). The other two claims follow similarly since the counit \(ER \xrightarrow[]{\varepsilon^R} \id_{\clC}\) of the adjunction pair \((E,R)\) is a natural isomorphism and also $EC$ is naturally isomorphic to $\id_{\clC}$ by Proposition~\ref{P: lift properties}, Part~3.

    \textit{Part 1.} This follows as \(L\) is a right exact functor.

    \textit{Part 2.} This follows as \(R\) is a left exact functor.

    \textit{Part 3.} By assumption, the sequence
    \[
    \clA(a,Cx) \xrightarrow[]{(Cf)_*} \clA(a,Cy) \xrightarrow[]{(Cg)_*} \clA(a,Cz)
    \]
    is exact for each object \(a\) in \(\clA\). Since the Yoneda Embedding \(t\mapsto \clA(-,t)\) reflects exactness \cite{weibel-introduction-1994}*{Yoneda Lemma~1.6.11}, it follows that the diagram
    \[
    0\xrightarrow[]{}Cx \xrightarrow[]{Cf} Cy \xrightarrow[]{Cg} Cz \xrightarrow[]{}0
    \]
    in \(\clA\) is exact at \(Cy\). Exactness at \(Cx\) and \(Cz\) follows from Proposition~\ref{P: lift properties}, Part~4 and Part~5. The second equivalence follows similarly.
\end{proof}

We can also classify when the intermediate extension functor is exact in terms of closure properties of its essential image.
    
\begin{prop}\label{P: C exact iff CC is sub quot closed}
        Consider Setup~\ref{Set: FF adj 3's}. Then the following statements are equivalent.
        \begin{enumerate}
            \item The intermediate extension functor \(C\) is exact.
            \item The essential image \(C\clC\) is closed under subobjects in \(\clA\).
            \item The essential image \(C\clC\) is closed under quotients in \(\clA\).
        \end{enumerate}
    \begin{proof}
        \((1\implies 2 \text{ and } 3)\): Let
        \begin{equation}\label{Eqn: sub and quot}
        0 \xrightarrow[]{} a \xrightarrow[]{f} Cy \xrightarrow[]{g} c \xrightarrow[]{} 0
        \end{equation}
        be a short exact sequence in \(\clA\), where \(y\) is an object of \(\clC\). We show that both \(a\) and \(c\) lie in the essential image \(C\clC\).

        Since \(E\) is exact and \(C\) is exact by assumption, applying \(CE\) to \eqref{Eqn: sub and quot} gives a short exact sequence
        \[
        0 \xrightarrow[]{} CEa \xrightarrow[]{CEf} CECy \xrightarrow[]{CEg} CEc \xrightarrow[]{} 0.
        \]
        By Proposition~\ref{P: lift properties}, Part~3, there is a natural isomorphism
        \[
        \beta^{-1}\coloneqq\id_\clC \xrightarrow[]{\eta^L}EL \xrightarrow[]{E\pi}EC.
        \]
        We first show that the composition
        \begin{equation}\label{eqn: gonna be zero}
        CEa \xrightarrow[]{CEf} CECy \xrightarrow[]{C\beta_y} Cy \xrightarrow[]{g} c 
        \end{equation}
        is zero.

        Using the right exactness of \(L\), the exactness of \(E\), and the naturality of the counit \(LE\xrightarrow[]{\varepsilon^L}\id_{\clA}\), we obtain a commutative diagram in \(\clA\)
        \[
        \begin{tikzcd}
            & LEa \arrow[r, "LEf"] \arrow[d, "\varepsilon^L_a"] & LECy \arrow[r, "LEg"] \arrow[d, "\varepsilon^L_{Cy}"] & LEc \arrow[r] \arrow[d, "\varepsilon^L_c"] & 0 \\
        0 \arrow[r] & a \arrow[r, "f"]                                  & Cy \arrow[r, "g"]                                     & c \arrow[r]                                & 0.
        \end{tikzcd}
        \]
        In particular, the composition
        \begin{equation}\label{eqn: is zero}
        LEa\xrightarrow[]{LEf} LECy \xrightarrow[]{\varepsilon^L_{Cy}} Cy \xrightarrow[]{g} c
        \end{equation}
        is zero.

        The composition \(Ea \xrightarrow[]{Ef} ECy \xrightarrow[]{\beta_y} y\) induces, by the naturality of \(\pi\), the naturality squares
        \begin{equation}\label{eqn: two naturality squares}
        \begin{tikzcd}
        LEa \arrow[d, "LEf"] \arrow[r, "\pi_{Ea}"]        & CEa \arrow[d, "CEf"]       \\
        LECy \arrow[d, "L\beta_y"] \arrow[r, "\pi_{ECy}"] & CECy \arrow[d, "C\beta_y"] \\
        Ly \arrow[r, "\pi_y"]                             & Cy.                     
        \end{tikzcd}
        \end{equation}
        We claim that
        \[
        \pi_yL(\beta_y)=\varepsilon^L_{Cy}.
        \]
        Indeed, the image of \(\pi_yL(\beta_y)\) under the natural isomorphisms
        \[
        \clA(LECy,Cy) \xrightarrow[]{E} \clC(ELECy,ECy) \xrightarrow[]{(\eta^L_{ECy})^*} \clC(ECy,ECy)
        \]
        is \(E(\pi_y)EL(\beta_y)\eta^L_{ECy}\), and we have
        \begin{align*}
        E(\pi_y)EL(\beta_y)\eta^L_{ECy}= E(\pi_y)\eta^L_y\beta_y = \beta_y^{-1}\beta_y = \id_{ECy},
        \end{align*}
        where the first equality follows from the naturality of \(\eta^L\), and the second equality follows from the definition of \(\beta\). This proves the claim.

        Using the claim, the vanishing of \eqref{eqn: is zero}, and the commutativity of \eqref{eqn: two naturality squares}, we obtain \(gC(\beta_y)CE(f)\pi_{Ea}=0\). By Lemma~\ref{L: N Fact, F(pi), F(i) iso}, Part~1, the morphism \(\pi_{Ea}\) is an epimorphism. Hence
        \[
        gC(\beta_y)CE(f)=0.
        \]
        Therefore, by the universal properties of the kernel and cokernel, there are unique morphisms \(\varphi^L_a\) and \(\varphi^L_c\) rendering the following diagram commutative:
        \[
        \begin{tikzcd}
        0 \arrow[r] & CEa \arrow[r, "CEf"] \arrow[d, "\varphi^L_a", dashed] & CECy \arrow[r, "CEg"] \arrow[d, "C\beta_y"] & CEc \arrow[r] \arrow[d, "\varphi^L_c", dashed] & 0 \\
        0 \arrow[r] & a \arrow[r, "f"]                                    & Cy \arrow[r, "g"]                         & c \arrow[r]                                    & 0.
        \end{tikzcd}
        \]
        Commutativity implies that \(\varphi^L_a\) is a monomorphism and \(\varphi^L_c\) is an epimorphism. Symmetrically, using \(R\) instead of \(L\), we obtain a commutative diagram
        \[
        \begin{tikzcd}
        0 \arrow[r] & a \arrow[r, "f"] \arrow[d, "\varphi^R_a", dashed] & Cy \arrow[r, "g"] \arrow[d, "C(\beta_y^{-1})"] & c \arrow[r] \arrow[d, "\varphi^R_c", dashed] & 0 \\
        0 \arrow[r] & CEa \arrow[r, "CEf"]                              & CECy \arrow[r, "CEg"]                          & CEc \arrow[r]                                & 0,
        \end{tikzcd}
        \]
        where \(\varphi^R_a\) is a monomorphism and \(\varphi^R_c\) is an epimorphism.

        Commutativity of the two diagrams gives
        \[
        f=f\varphi^L_a\varphi^R_a
        \qquad\text{and}\qquad
        g=\varphi^L_c\varphi^R_cg.
        \]
        Since \(f\) is a monomorphism and \(g\) is an epimorphism, it follows that
        \[
        \varphi^L_a\varphi^R_a=\id_a
        \qquad\text{and}\qquad
        \varphi^L_c\varphi^R_c=\id_c.
        \]
        Hence \(\varphi^L_a\) and \(\varphi^R_c\) are isomorphisms. Therefore \(a\cong CEa\) and \(c\cong CEc\). Thus both \(a\) and \(c\) lie in the essential image \(C\clC\), as required.

        \((2\implies 1)\): Let
        \[
        0 \xrightarrow[]{} x \xrightarrow[]{f} y \xrightarrow[]{g} z \xrightarrow[]{}0
        \]
        be a short exact sequence in \(\clC\). We show that the sequence
        \begin{equation}\label{eqn: C(exact sequence)}
        0 \xrightarrow[]{} Cx \xrightarrow[]{Cf} Cy \xrightarrow[]{Cg} Cz \xrightarrow[]{} 0
        \end{equation}
        is exact in \(\clA\).

        By Proposition~\ref{P: lift properties}, Part~4, the morphism \(Cg\) is an epimorphism. Hence we have a short exact sequence
        \[
        0 \xrightarrow[]{} \ker(Cg) \xrightarrow[]{} Cy \xrightarrow[]{Cg} Cz \xrightarrow[]{} 0.
        \]
        Since the essential image \(C\clC\) is closed under subobjects, there is an object \(x'\) of \(\clC\) and an isomorphism \(\ker(Cg)\cong Cx'\). Since \(C\) is fully faithful by Proposition~\ref{P: lift properties}, Part~2, the above short exact sequence is isomorphic to a short exact sequence of the form
        \[
        0 \xrightarrow[]{} Cx' \xrightarrow[]{C(f')} Cy \xrightarrow[]{Cg} Cz \xrightarrow[]{} 0,
        \]
        for a unique morphism \(x' \xrightarrow[]{f'}y\). The intermediate extension functor \(C\) reflects short exact sequences by Proposition~\ref{P: reflect ses}. Hence
        \[
        0 \xrightarrow[]{} x' \xrightarrow[]{f'} y \xrightarrow[]{g} z \xrightarrow[]{} 0
        \]
        is exact in \(\clC\).

        Now \(x \xrightarrow[]{f} y\) and \(x' \xrightarrow[]{f'} y\) are both kernels of \(g\). Therefore there is a unique isomorphism \(x \xrightarrow[]{u} x'\) rendering the following diagram commutative:
        \[
        \begin{tikzcd}
        0 \arrow[r] & x \arrow[r, "f"] \arrow[d, "u"] & y \arrow[r, "g"] \arrow[d, equal] & z \arrow[r] \arrow[d, equal] & 0 \\
        0 \arrow[r] & x' \arrow[r, "f'"]                       & y \arrow[r, "g"]                  & z \arrow[r]                  & 0.
        \end{tikzcd}
        \]
        Applying \(C\) to this diagram shows that \eqref{eqn: C(exact sequence)} is exact.

        \((3\implies 1)\): This follows by a similar argument to the previous implication.
    \end{proof}
\end{prop}

    \begin{cor}\label{C: C exact iff C(C) is Serre sub}
    Consider Setup~\ref{Set: FF adj 3's}. The intermediate extension functor \(C\) is exact if and only if its essential image \(C\clC\) is a Serre subcategory of \(\clA\).
    \begin{proof}
        This follows from Proposition~\ref{P: LB,RB, LiftB extension closed in A}, Part 3 and Proposition~\ref{P: C exact iff CC is sub quot closed}.
    \end{proof}
    \end{cor}

    The preceding results yield three new codistinguished embeddings from a given one, together with criteria for determining when they are proper.

\begin{thm}\label{T: three codistinguished embeddings}
    Consider Setup~\ref{Set: FF adj 3's}. Let $\clA$ be a codistinguished subcategory of $\clT$. Then each of the three functors $L$, $R$ and $C$ realises $\clC$ as a codistinguished abelian subcategory of $\clT$. Moreover, if $\clA$ is extension closed in $\clT$, then the essential images $L\clC$, $R\clC$ and $C\clC$ are also extension closed in $\clT$.
\end{thm}

\begin{cor}\label{C: three proper embeddings}
    Consider Setup~\ref{Set: FF adj 3's}. Let \(\clA\) be a proper abelian subcategory of a triangulated category \(\clT\). Then each of the following holds:
    \begin{itemize}
        \item[1.] $L$ realises $\clC$ as a proper abelian subcategory of $\clT$ if and only if $L$ preserves monomorphisms.
        \item[2.] $R$ realises $\clC$ as a proper abelian subcategory of $\clT$ if and only if $R$ preserves epimorphisms.
        \item[3.] $C$ realises $\clC$ as a proper abelian subcategory of $\clT$ if and only if the essential image $C\clC$ is a Serre subcategory of $\clA$.
    \end{itemize}
\end{cor}

\section{Idempotent subalgebras}\label{S: Idempotent subalgebras}

We illustrate the construction of Section~\ref{S: Codistinguished from FF} using idempotent subalgebras. This produces a large class of examples motivated by recollements of module categories: indeed, every recollement whose middle term is the module category of a semiprimary ring is equivalent to one induced by an idempotent element; see \cite{psaroudakis-recollements-2014}*{Cor.\ 5.5}. We then treat Auslander algebras and their associated finite representation type algebras in more detail.
    
    Let \(\Gamma\) be a ring and let \(e\) be an idempotent element in \(\Gamma\). Then \(e\) induces a fully faithful adjunction triple
        \begin{equation}\label{diag: idem recollements}
        \begin{tikzcd}
            \mod{\Gamma} \arrow[rr, "-\otimes_{\Gamma}\Gamma e" description] 
            &  & 
            \mod{e\Gamma e} \arrow[ll, "-\otimes_{e\Gamma e} e\Gamma"', bend right] 
            \arrow[ll, "{\Hom_{e\Gamma e}(\Gamma e, -)}", bend left]
        \end{tikzcd}
        \end{equation}
    where the functor \(-\otimes_{\Gamma}\Gamma e\) is naturally isomorphic to both \(\Hom_{\Gamma}(e\Gamma,-)\) and \((-)e\). This adjunction triple appears on the right-hand side of the recollement associated to \(e\).

    \begin{prop}\label{P: Norm Idemp}
    Let \(T\) be an \(e\Gamma e\)-module. Then the \(\Gamma\)-module homomorphism
        \begin{align*}
            T\otimes_{e\Gamma e}e\Gamma &\xrightarrow[]{\mathscr{n}_T} \Hom_{e\Gamma e}(\Gamma e,T),\\
            t\otimes er &\mapsto [se\mapsto terse],
        \end{align*}
        is the component at \(T\) of the norm constructed in Lemma~\ref{L: Norm_fully faith adjunction triple} associated to the fully faithful adjunction triple \eqref{diag: idem recollements}.
    \begin{proof}
        By Lemma~\ref{L: Norm_fully faith adjunction triple}, Part~1, it suffices to show that \(\mathscr{n}_T\) corresponds to \(\id_T\) under the diagonal natural isomorphism
        \[
        \Hom_\Gamma(T\otimes_{e\Gamma e}e\Gamma,\Hom_{e\Gamma e}(\Gamma e,T))
        \xrightarrow[]{\Phi_T}
        \Hom_{e\Gamma e}(T,T).
        \]
        By definition, \(\Phi_T^{-1}\) is the composition of the inverse of the unit of the adjunction
        \[
        (-\otimes_{e\Gamma e}e\Gamma,-\otimes_\Gamma\Gamma e),
        \]
        followed by the Tensor-Hom adjunction isomorphism associated to
        \[
        (-\otimes_\Gamma\Gamma e,\Hom_{e\Gamma e}(\Gamma e,-)).
        \]
        Thus we consider the composition of natural isomorphisms
        \begin{align*}
        \Hom_{e\Gamma e}(T,T)
            &\xrightarrow[]{}
            \Hom_{e\Gamma e}(T\otimes_{e\Gamma e}e\Gamma e,T)  \\
            &\xrightarrow[]{}
            \Hom_{e\Gamma e}(T\otimes_{e\Gamma e} e\Gamma\otimes_\Gamma \Gamma e,T) \\
            &\xrightarrow[]{}
            \Hom_{\Gamma}(T\otimes_{e\Gamma e}e\Gamma,\Hom_{e\Gamma e}(\Gamma e,T)).
        \end{align*}
        The first isomorphism is induced by precomposition with the unitor
        \[
        T\otimes_{e\Gamma e}e\Gamma e\xrightarrow[]{}T.
        \]
        The second isomorphism is induced by precomposition with the \(e\Gamma e\)-module isomorphism
        \begin{align*}
        T\otimes_{e\Gamma e} e\Gamma \otimes_\Gamma \Gamma e &\xrightarrow{} T\otimes_{e\Gamma e} e\Gamma e,\\
        t\otimes er \otimes se &\mapsto t\otimes erse.
        \end{align*}
        The composition of these two \(e\Gamma e\)-module isomorphisms is precisely the inverse of the unit of the adjunction \((-\otimes_{e\Gamma e}e\Gamma,-\otimes_\Gamma\Gamma e)\). The final isomorphism is the Tensor-Hom adjunction isomorphism. Under this composition, the identity \(\id_T\) corresponds exactly to the \(\Gamma\)-module homomorphism
        \begin{align*}
            T\otimes_{e\Gamma e}e\Gamma &\xrightarrow[]{\mathscr{n}_T} \Hom_{e\Gamma e}(\Gamma e,T),\\
            t\otimes er &\mapsto [se\mapsto terse],
        \end{align*}
        as required.
    \end{proof}
    \end{prop}

    We now begin our treatment of Auslander algebras and their associated finite representation type algebras.

    \begin{setup}\label{Setup: Auslander-Idempotent}
        Let \(A\) be a finite dimensional algebra over a field \(k\), and suppose that \(A\) is representation finite. Let \(M=A\oplus \xoverline{M}\) be a basic additive generator of \(\mod{A}\), and set
        \[
        \Gamma=\End_A(M),
        \]
        the Auslander algebra of \(A\). Let \(e\) in \(\Gamma\) be the idempotent
        \[
        e\colon M\xrightarrow[]{\pi_A}A \xrightarrow[]{\iota_A}M,
        \]
        where \(\pi_A\) denotes the projection onto \(A\) and \(\iota_A\) denotes the inclusion of \(A\). Then there are isomorphisms
        \[
        e\Gamma e\cong A, 
        \qquad 
        e\Gamma \cong \Hom_A(M,A)
        \qquad\text{and}\qquad 
        \Gamma e\cong M,
        \]
        of algebras, \((A,\Gamma)\)-bimodules and \((\Gamma,A)\)-bimodules, respectively.

        With these identifications, the fully faithful adjunction triple \eqref{diag: idem recollements} becomes
        \begin{equation}\label{diag: FF Auslander}
        \begin{tikzcd}[column sep=large]
            \mod\Gamma \arrow[rr, "(-)e" description] 
            &  & 
            {\mod{A}.} 
            \arrow[ll, "L\coloneqq{-\otimes_{A}\Hom_A(M,A)}"', bend right] 
            \arrow[ll, "{R\coloneqq\Hom_A(M,-)}", bend left]
        \end{tikzcd}
        \end{equation}
        Using Proposition~\ref{P: Norm Idemp}, the component of the associated norm is the Tensor-Evaluation map
        \begin{align*}
        T\otimes_A\Hom_A(M,A) &\xrightarrow[]{\mathscr{n}_T} \Hom_A(M,T),\\
        t\otimes f &\mapsto [m\mapsto tf(m)].
        \end{align*}
        
        We have the intermediate extension functor
        \[
        \mod{A} \xrightarrow[]{C} \mod\Gamma.
        \]
    \end{setup}

    \begin{rem}\label{R: Auslander algebra revisited}
        Consider \(\mod\Gamma\) as the heart of the canonical \(t\)-structure on \(\bDer(\mod\Gamma)\). Theorem~\ref{T: three codistinguished embeddings} and diagram~\eqref{diag: FF Auslander} show that \(\mod{A}\) can be realised, in three possibly distinct ways, as a codistinguished abelian subcategory of \(\bDer(\mod\Gamma)\). Moreover, since \(A\) is representation finite, the functor \(R\) induces an equivalence
        \[
        \add(M)=\mod{A}\xrightarrow[]{\simeq}\proj\Gamma,
        \]
        for \(M\) a basic additive generator of \(\mod{A}\). Therefore, Example~\ref{E: Auslander algebras2} is a special case of Theorem~\ref{T: three codistinguished embeddings}.
    \end{rem}

    Unlike \(L\) and \(R\), the intermediate extension functor \(C\) is not given explicitly from the outset. We therefore collect some lemmas that make it easier to work with. The first gives control over its action on projective \(A\)-modules.
        
    \begin{lem}\label{L: Idem Lift-Proj}
        Consider Setup~\ref{Setup: Auslander-Idempotent}. Let \(P\) be an \(A\)-module and let
        \[
        LP \xrightarrow[]{\mathscr{n}_P} RP
        \]
        be the component of the norm associated to the fully faithful adjunction triple \eqref{diag: FF Auslander}. If \(P\) is finitely generated and projective, then the following statements hold:
        \begin{enumerate}
            \item \(\mathscr{n}_P\) is an isomorphism.
            \item The components
                \[
                LP\xrightarrow[]{\pi_P} CP 
                \qquad \text{and} \qquad 
                CP \xrightarrow[]{\iota_P}RP
                \]
                of the canonical natural transformations obtained in Lemma~\ref{L: N Fact, F(pi), F(i) iso} are isomorphisms.
        \end{enumerate}
        \begin{proof}
        \textit{Part 1.} See, for example, \cite{christensen-derived-2024}*{Prop.\ 1.4.6(a)}.

        \textit{Part 2.} This follows from Part~1 and Lemma~\ref{L: N Fact, F(pi), F(i) iso}.
        \end{proof}
        \end{lem}
        
        The next lemma gives an explicit description of the intermediate extension functor.
        
        \begin{lem}\label{L: Idemp Lift explicit}
        Consider Setup~\ref{Setup: Auslander-Idempotent}. Let \(T\) be an \(A\)-module and let
            \[
            LT \xrightarrow[]{\mathscr{n}_T} RT
            \]
            be the component of the norm associated to the fully faithful adjunction triple \eqref{diag: FF Auslander}. Then the \(\Gamma\)-module \(CT\) identifies with
            \[
            CT
            =
            \{ f \in \Hom_A(M,T) \mid f \text{ factors through a finitely generated free } A\text{-module}\}.
            \]
        \begin{proof}
        By definition, \(CT\) is the image of the norm map
        \[
        LT=T\otimes_A\Hom_A(M,A)
        \xrightarrow[]{\mathscr{n}_T}
        \Hom_A(M,T)=RT,
        \]
        which we regard as a submodule of \(\Hom_A(M,T)\). Let \(t\otimes f\) be an elementary tensor in \(T\otimes_A\Hom_A(M,A)\). Then \(\mathscr{n}_T(t\otimes f)\) factors as
        \[
        M \xrightarrow[]{f} A \xrightarrow[]{\lambda_t} T,
        \]
        where \(\lambda_t\) is the \(A\)-module homomorphism determined by \(1_A\mapsto t\). Thus the image of every elementary tensor factors through the finitely generated free module \(A\).
        
        Now let \(x=\sum_{i=1}^n t_i\otimes f_i\) be an arbitrary element of \(T\otimes_A\Hom_A(M,A)\). Then
        \[
        \mathscr{n}_T(x)=\sum_{i=1}^n \mathscr{n}_T(t_i\otimes f_i),
        \]
        and this morphism factors as
        \[
        M \xrightarrow[]{\begin{pmatrix}
        f_1\\
        \vdots\\
        f_n
        \end{pmatrix}} A^n
        \xrightarrow[]{\begin{pmatrix}
        \lambda_{t_1} & \cdots & \lambda_{t_n}
        \end{pmatrix}} T.
        \]
        Hence every element of \(CT=\img(\mathscr{n}_T)\) factors through a finitely generated free \(A\)-module.
                
        Conversely, suppose that \(M\xrightarrow[]{h}T\) factors through a finitely generated free \(A\)-module. Thus we may write \(h\) as a composite
        \[
        M \xrightarrow[]{\begin{pmatrix}
        f_1\\
        \vdots\\
        f_n
        \end{pmatrix}} A^n
        \xrightarrow[]{\begin{pmatrix}
        \lambda_{t_1} & \cdots & \lambda_{t_n}
        \end{pmatrix}} T,
        \]
        for suitable \(A\)-module homomorphisms \(M\xrightarrow[]{f_i} A\). With this, we get
        \[
        h=\mathscr{n}_T\Big(\sum_{i=1}^n t_i\otimes f_i\Big). \qedhere
        \]
        \end{proof}
        \end{lem}

        We now analyse when the functors \(L\), \(R\), and \(C\) are exact. This determines when the associated embeddings from Corollary~\ref{C: three proper embeddings} are proper, rather than merely codistinguished.

        \begin{prop}\label{P: Idem R exact}
            Consider Setup~\ref{Setup: Auslander-Idempotent}. The following statements are equivalent:
            \begin{enumerate}
            \item \(R\) is exact.
            \item \(\gdim{A}=0\).
            \item The canonical natural transformations
                \[
                L\xrightarrow[]{\pi}C
                \qquad\text{and}\qquad
                C\xrightarrow[]{\iota}R
                \]
                obtained in Lemma~\ref{L: N Fact, F(pi), F(i) iso} are natural isomorphisms.
            \end{enumerate}
        \begin{proof}
            \((1)\Leftrightarrow{}(2)\): The functor \(R=\Hom_A(M,-)\) is exact if and only if \(M\) is a projective \(A\)-module. Since \(M\) is an additive generator of \(\mod A\), this holds if and only if every finitely generated \(A\)-module is projective, equivalently if and only if \(A\) is semisimple. This is equivalent to \(\gdim{A}=0\).
            
            \((2)\xRightarrow{}(3)\): If \(\gdim{A}=0\), then every finitely generated \(A\)-module is projective. The statement therefore follows from Lemma~\ref{L: Idem Lift-Proj}.

            \((3)\Rightarrow(1)\): The functor \(R\) is left exact. By assumption, \(R\) is naturally isomorphic to \(L\), and \(L\) is right exact. Hence \(R\) is also right exact and therefore is exact.
        \end{proof}
        \end{prop}

        The converse of Part~2 of the following result will be obtained in Corollary~\ref{C: L lift iso iff hereditary}, after establishing the exactness criterion for \(C\).

        \begin{prop}\label{P: Idem L exact}
        Consider Setup~\ref{Setup: Auslander-Idempotent}. The following statements hold:
            \begin{enumerate}
            \item \(L\) is exact if and only if \(\gdim{A}\leq2\).
            
            \item If \(\gdim{A}\leq1\), then the canonical natural transformation
                \[
                L\xrightarrow[]{\pi}C
                \]
                obtained in Lemma~\ref{L: N Fact, F(pi), F(i) iso} is a natural isomorphism.
            \end{enumerate}
        \begin{proof}
        \textit{Part 1.} Recall that \(M=A\oplus \xoverline{M}\) is a basic additive generator of \(\mod{A}\). The functor \(L\) is exact if and only if \(\Hom_A(M,A)\) is flat as an \(A^{op}\)-module. Since \(A\) is finite dimensional, this is equivalent to \(\Hom_A(M,A)\) being projective as an \(A^{op}\)-module; see \cite{christensen-derived-2024}*{Thm.\ 1.3.48}. The functor \(\Hom_A(-,A)\) induces an equivalence
        \[
            (\proj{A})^{op}\simeq\proj{A^{op}}.
        \]
        Hence \(L\) is exact if and only if \(\Hom_A(\xoverline{M},A)\) is projective as an \(A^{op}\)-module. Equivalently, we need to determine when \(\Hom_A(X,A)\) is projective as an \(A^{op}\)-module for every indecomposable non-projective \(A\)-module \(X\).
            
        \((\xLeftarrow[]{\text{if}})\): Let \(X\) be an indecomposable non-projective \(A\)-module, and let
        \[
            P_1\xrightarrow[]{f}P_0 \xrightarrow[]{} X \xrightarrow[]{} 0
        \]
        be a minimal projective presentation of \(X\). Applying \(\Hom_A(-,A)\), we obtain a minimal projective presentation
        \[
            \Hom_A(P_0,A) 
            \xrightarrow[]{\Hom_A(f,A)} 
            \Hom_A(P_1,A) 
            \xrightarrow[]{} 
            \coker\Hom_A(f,A) 
            \xrightarrow[]{} 0
        \]
        of \(\coker\Hom_A(f,A)\) in \(\mod A^{op}\); see \cite{assem-elements1-2006}*{Prop.\ 2.1(b)}. This extends to an exact sequence
        \begin{equation}\label{eqn: Exact sequence}
            0 \xrightarrow[]{} \Hom_A(X,A) 
            \xrightarrow[]{} \Hom_A(P_0,A) 
            \xrightarrow[]{} \Hom_A(P_1,A) 
            \xrightarrow[]{} \coker\Hom_A(f,A) 
            \xrightarrow[]{} 0.
        \end{equation}
        If \(\gdim{A}\leq2\), then \(\coker\Hom_A(f,A)\) has projective dimension at most \(2\) as an \(A^{op}\)-module. Hence its second syzygy \(\Hom_A(X,A)\) is projective. Thus \(\Hom_A(X,A)\) is projective for every indecomposable non-projective \(A\)-module \(X\), and so \(L\) is exact.
            
        \((\xRightarrow[]{\text{only if}})\): Conversely, suppose that \(\Hom_A(X,A)\) is projective as an \(A^{op}\)-module for every indecomposable non-projective \(A\)-module \(X\). Applying the standard dual \(D=\Hom_k(-,k)\) to \eqref{eqn: Exact sequence} yields an exact sequence
        \[
            0\xrightarrow[]{} \tau X 
            \xrightarrow[]{} \nu P_1 
            \xrightarrow[]{} \nu P_0 
            \xrightarrow[]{} \nu X 
            \xrightarrow[]{} 0,
        \]
        where \(\tau\) denotes the Auslander-Reiten translate and \(\nu\) denotes the Nakayama functor. Since \(\Hom_A(X,A)\) is projective as an \(A^{op}\)-module, the module \(\nu X=D\Hom_A(X,A)\) is injective. Hence this sequence is an augmented injective coresolution of \(\tau X\), and therefore
        \[
            \idim_A{\tau X}\leq2.
        \]
        Since \(\tau\) induces a bijection between the indecomposable non-projective and indecomposable non-injective \(A\)-modules, every indecomposable \(A\)-module has injective dimension at most \(2\). It follows that \(\gdim A\leq2\).

        \textit{Part 2.} By Lemma~\ref{L: N Fact, F(pi), F(i) iso}, Part~1, it suffices to show that the component of the norm
        \[
            LX \xrightarrow[]{\mathscr{n}_X} RX
        \]
        is a monomorphism for every finitely generated \(A\)-module \(X\).

        If \(X\) is projective, then this follows from Lemma~\ref{L: Idem Lift-Proj}. Suppose then that \(X\) is non-projective. Since \(\gdim{A}\leq1\), there is an augmented projective resolution
        \[
            0 \xrightarrow[]{} P_1 \xrightarrow[]{} P_0 \xrightarrow[]{} X \xrightarrow[]{} 0.
        \]
        Using the right exactness of \(L\), the left exactness of \(R\), and the naturality of the norm \(\mathscr{n}\), we obtain a commutative diagram in \(\mod\Gamma\)
        \[
        \begin{tikzcd}
                & LP_1 \arrow[r] \arrow[d, "\mathscr{n}_{P_1}"] 
                & LP_0 \arrow[r] \arrow[d, "\mathscr{n}_{P_0}"] 
                & LX \arrow[r] \arrow[d, "\mathscr{n}_{X}"] 
                & 0 \\
                0 \arrow[r] 
                & RP_1 \arrow[r]                                    
                & RP_0 \arrow[r]                                    
                & RX.                                          
        \end{tikzcd}
        \]
        The rows are exact, and by Lemma~\ref{L: Idem Lift-Proj}, Part~1, the maps \(\mathscr{n}_{P_0}\) and \(\mathscr{n}_{P_1}\) are isomorphisms. It follows from the Snake Lemma (or if you prefer from the Four Lemma) that \(\mathscr{n}_X\) is a monomorphism. Hence \(L\xrightarrow[]{\pi}C\) is a natural isomorphism.
    \end{proof}
    \end{prop}

    \begin{prop}\label{P: Idem lift exact}
        Consider Setup~\ref{Setup: Auslander-Idempotent}. The intermediate extension functor \(C\) is exact if and only if \(\gdim{A}\leq1\).
    \begin{proof}
    \((\xLeftarrow[]{\text{if}})\): Assume that \(\gdim{A}\leq 1\). 
    By Proposition~\ref{P: Idem L exact}, Part~2, the canonical natural transformation
    \[
    L \xrightarrow[]{\pi} C
    \]
    is a natural isomorphism. Since \(L\) is exact by Proposition~\ref{P: Idem L exact}, Part~1, it follows that
    \(C\) is exact.

    \((\xRightarrow[]{\text{only if}})\): The rightmost adjoint \(R\) provides us with an answer to this implication. Conversely, suppose that \(A\) is not hereditary, and assume for a contradiction that \(C\) is exact. Choose an \(A\)-module \(X\) with \(\pdim X>1\), together with a short exact sequence
    \[
        0 \xrightarrow[]{} Z \xrightarrow[]{f} P \xrightarrow[]{g} X \xrightarrow[]{} 0,
    \]
    where \(P\) is projective and \(Z\) is not projective. Applying \(C\), we obtain a short exact sequence of \(\Gamma\)-modules
    \[
        0 \xrightarrow[]{} CZ 
        \xrightarrow[]{Cf} 
        CP 
        \xrightarrow[]{Cg} 
        CX 
        \xrightarrow[]{} 0.
    \]
    We construct the following diagram:
    \[
    \begin{tikzcd}
        &                                                                     & RZ \arrow[d, "Rf"] \arrow[ldd, "h"', dashed]                        &                                                 &   \\
        &                                                                     & RP \arrow[d, "\iota^{-1}_P"]                                        &                                                 &   \\
        0 \arrow[r] & CZ \arrow[r, "Cf"] \arrow[d, "\iota_Z"] & CP \arrow[r, "Cg"] \arrow[d, "\iota_P"] & CX \arrow[d, "\iota_X"] \arrow[r] & 0 \\
        0 \arrow[r] & RZ \arrow[r, "Rf"]                                                  & RP \arrow[r, "Rg"]                                                  & RX.                                              &  
    \end{tikzcd}
    \]
    Since \(P\) is projective, the component \(CP \xrightarrow[]{\iota_P} RP\) is an isomorphism by Lemma~\ref{L: Idem Lift-Proj}. Consider the composite
    \[
        t \coloneqq 
        RZ \xrightarrow[]{Rf} RP 
        \xrightarrow[]{\iota_P^{-1}} CP.
    \]
    We claim that \(Cg\circ t=0\). Indeed, postcomposing with \(\iota_X : CX \xrightarrow[]{} RX \) gives
    \[
        \iota_X \circ Cg \circ t
         = Rg \circ \iota_P \circ \iota_P^{-1} \circ Rf
         = Rg \circ Rf
         = 0.
    \]
    Since \(\iota_X\) is a monomorphism, it follows that \(Cg\circ t=0\).
        
    By exactness, there exists a unique \(\Gamma\)-module homomorphism \(RZ \xrightarrow[]{h} CZ\) such that \(Cf\circ h=t\). Hence
    \[
        Rf \circ \iota_Z \circ h
        = \iota_P \circ Cf \circ h
        = \iota_P \circ t
        = \iota_P \circ \iota_P^{-1} \circ Rf
        = Rf.
    \]
    Since \(Rf\) is a monomorphism, we conclude that \(\iota_Z h = \id_{RZ}\). Thus \(\iota_Z\) is a retraction. Since \(\iota_Z\) is also a monomorphism, it is an isomorphism. We therefore have an isomorphism
    \begin{equation}\label{eqn: factor through projectives}
        CZ \xrightarrow[]{\ \iota_Z\ } RZ = \Hom_A(M,Z).
    \end{equation}

    By Lemma~\ref{L: Idemp Lift explicit}, the \(\Gamma\)-module \(CZ\) identifies with the submodule of \(\Hom_A(M,Z)\) consisting of those morphisms \(M\xrightarrow[]{} Z\) that factor through a finitely generated free \(A\)-module. Hence \eqref{eqn: factor through projectives} implies that every morphism \(M\xrightarrow[]{} Z\) factors through a finitely generated free \(A\)-module.
        
    Since \(M\) is an additive generator of \(\mod A\), there exist a natural number \(n\) and morphisms \(Z \xrightarrow[]{i} M^n\) and \(M^n \xrightarrow[]{p} Z\) such that
     \[
     p i = \id_Z.
    \]
    For each \(1\leq j\leq n\), let
    \[
    p_j \coloneqq 
    M \xrightarrow[]{\iota_j} M^n \xrightarrow[]{p} Z,
    \]
    where \(M\xrightarrow[]{\iota_j} M^n\) denotes the \(j\)-th canonical inclusion. By the previous paragraph, each \(p_j\) factors through a finitely generated free \(A\)-module \(F_j\). Hence \(p\) factors through the finitely generated free module
    \[
    F \coloneqq \bigoplus_{j=1}^n F_j.
    \]
    Therefore \(\id_Z=p i\) factors through \(F\). It follows that \(Z\) is a direct summand of a projective module, and hence \(Z\) is projective. This contradicts our choice of \(Z\). Therefore \(C\) is not exact whenever \(\gdim{A}>1\). Hence \(C\) is exact if and only if \(\gdim{A}\leq1\).
    \end{proof}
    \end{prop}

    \begin{cor}\label{C: L lift iso iff hereditary}
        Consider Setup~\ref{Setup: Auslander-Idempotent}. The canonical natural transformation
    \[
        L\xrightarrow[]{\pi}C
    \]
    obtained in Lemma~\ref{L: N Fact, F(pi), F(i) iso} is a natural isomorphism if and only if
    \(\gdim{A}\leq 1\).
    \begin{proof}
    \((\xLeftarrow{\text{if}})\): The result follows from
    Proposition~\ref{P: Idem L exact}, Part~2.

    \((\xRightarrow{\text{only if}})\): Suppose that \(L\xrightarrow[]{\pi}C\) is a natural isomorphism. Since \(L\) is right exact, it follows that
    \(C\) is right exact. Moreover, by
    Proposition~\ref{P: lift properties}, Part~5, the functor
    \(C\) preserves monomorphisms. Hence
    \(C\) is exact. Therefore, by
    Proposition~\ref{P: Idem lift exact}, we have \(\gdim{A}\leq 1\).
    \end{proof}
    \end{cor}

    We summarise the preceding discussion in the following table.

    \begin{table}[ht]
    \centering
        \setlength{\extrarowheight}{6pt}
        \setlength{\tabcolsep}{10pt}
        \setlength{\arrayrulewidth}{0.6pt}
        \arrayrulecolor{black}
        \begin{tabular}{|>{\columncolor[HTML]{EFEFEF}}c|c|c|}
        \hline
        \cellcolor[HTML]{C0C0C0}\textbf{Functor} 
        & \cellcolor[HTML]{C0C0C0}\textbf{Exact if and only if} 
        & \cellcolor[HTML]{C0C0C0}\textbf{Natural equivalences} \\ 
        \hline
        \(R\) 
        & \(\gdim{A}=0\) 
        & \(\gdim{A}=0 \iff L \cong C \cong R\) \\ 
        \hline
        \(L\) 
        & \(\gdim{A}\leq 2\) 
        & \(\gdim{A}\leq 1 \iff L \cong C\) \\ 
        \hline
        \(C\) 
        & \(\gdim{A}\leq 1\) 
        & --- \\ 
        \hline
        \end{tabular}
    \caption{Exactness and comparison criteria for the functors in Setup~\ref{Setup: Auslander-Idempotent}.}
    \label{tab: Auslander exactness comparison}
    \end{table}
        
        We now give an explicit example in which \(\gdim{A}=\infty\). In this case, none of the three functors \(L\), \(C\), and \(R\) is exact. The example therefore illustrates the different ways in which exactness can fail.
        
        \begin{ex}
        Let \(k\) be a field. Consider the quiver
        \[
        \begin{tikzcd}
        Q\colon 1 \arrow[r, shift left] & 2 \arrow[l, shift left]
        \end{tikzcd}
        \]
        and let \(I\) be the admissible ideal generated by the path
        \[
        2 \longrightarrow 1 \longrightarrow 2.
        \]
        Set \(\Gamma=kQ/I\). The Auslander--Reiten quiver of the bound quiver algebra \(\Gamma\) is
        
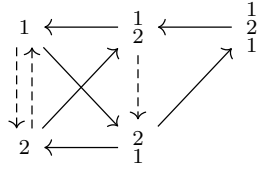
\[
        \begin{tikzcd}
        \substack{1} \arrow[rd] \arrow[d, dashed, shift right] & \substack{1\\2} \arrow[l] \arrow[d, dashed] & \substack{1\\2\\1} \arrow[l] \\
        \substack{2} \arrow[ru] \arrow[u, dashed, shift right] & \substack{2\\1} \arrow[l] \arrow[ru]        &                             
        \end{tikzcd}
        \]
        The algebra \(\Gamma\) is the Auslander algebra of the dual numbers \(A=k[x]/(x^2)\). We identify \(A\) with the idempotent subalgebra associated to the idempotent \(e_1\in \Gamma\), so that
        \[
        A=e_1\Gamma e_1.
        \]
        
        The following three copies of the Auslander--Reiten quiver of \(\Gamma\) highlight the essential images of \(L\), \(C\), and \(R\), respectively.
        
        
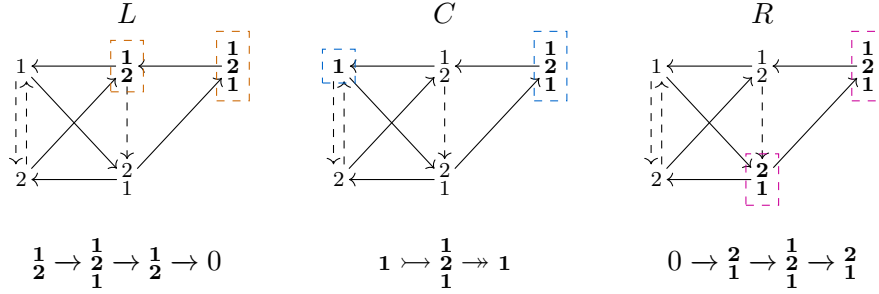
\begin{center}
        \begin{tikzpicture}[>={Computer Modern Rightarrow}, every node/.style={inner sep=1.5pt}, font=\small]
        
        \node at (1.4,1.45) {\(L\)};
        
        \node (L2)   at (0,0.75)    {\(\substack{1}\)};
        \node (L21)  at (1.4,0.75)  {\(\boldsymbol{\substack{1\\2}}\)};
        \node (L212) at (2.8,0.75)  {\(\boldsymbol{\substack{1\\2\\1}}\)};
        \node (L1)   at (0,-0.75)   {\(\substack{2}\)};
        \node (L12)  at (1.4,-0.75) {\(\substack{2\\1}\)};
        
        \draw[->] (L21) -- (L2);
        \draw[->] (L212) -- (L21);
        \draw[->] (L12) -- (L1);
        
        \draw[->] (L2) -- (L12);
        \draw[->] (L1) -- (L21);
        \draw[->] (L12) -- (L212);
        
        \draw[->, dashed] (-0.07,0.55) -- (-0.07,-0.55);
        \draw[->, dashed] (0.07,-0.55) -- (0.07,0.55);
        \draw[->, dashed] (L21) -- (L12);
        
        \node[draw=myorange, dashed, fit=(L212), inner sep=2pt] {};
        \node[draw=myorange, dashed, fit=(L21), inner sep=2pt] {};
        
        \node at (1.4,-1.85) {\(\boldsymbol{\substack{1\\2}} \xrightarrow[]{} \boldsymbol{\substack{1\\2\\1}} \xrightarrow[]{} \boldsymbol{\substack{1\\2}} \xrightarrow[]{} 0\)};
        
        \node at (5.6,1.45) {\(C\)};
        
        \node (M2)   at (4.2,0.75)    {\(\boldsymbol{\substack{1}}\)};
        \node (M21)  at (5.6,0.75)    {\(\substack{1\\2}\)};
        \node (M212) at (7.0,0.75)    {\(\boldsymbol{\substack{1\\2\\1}}\)};
        \node (M1)   at (4.2,-0.75)   {\(\substack{2}\)};
        \node (M12)  at (5.6,-0.75)   {\(\substack{2\\1}\)};
        
        \draw[->] (M21) -- (M2);
        \draw[->] (M212) -- (M21);
        \draw[->] (M12) -- (M1);
        
        \draw[->] (M2) -- (M12);
        \draw[->] (M1) -- (M21);
        \draw[->] (M12) -- (M212);
        
        \draw[->, dashed] (4.13,0.55) -- (4.13,-0.55);
        \draw[->, dashed] (4.27,-0.55) -- (4.27,0.55);
        \draw[->, dashed] (M21) -- (M12);
        
        \node[draw=myblue, dashed, fit=(M212), inner sep=2pt] {};
        \node[draw=myblue, dashed, fit=(M2), inner sep=2pt] {};
        
        \node at (5.6,-1.85) {\(\boldsymbol{\substack{1}} \rightarrowtail \boldsymbol{\substack{1\\2\\1}} \twoheadrightarrow \boldsymbol{\substack{1}}\)};
        
        \node at (9.8,1.45) {\(R\)};
        
        \node (R2)   at (8.4,0.75)    {\(\substack{1}\)};
        \node (R21)  at (9.8,0.75)    {\(\substack{1\\2}\)};
        \node (R212) at (11.2,0.75)   {\(\boldsymbol{\substack{1\\2\\1}}\)};
        \node (R1)   at (8.4,-0.75)   {\(\substack{2}\)};
        \node (R12)  at (9.8,-0.75)   {\(\boldsymbol{\substack{2\\1}}\)};
        
        \draw[->] (R21) -- (R2);
        \draw[->] (R212) -- (R21);
        \draw[->] (R12) -- (R1);
        
        \draw[->] (R2) -- (R12);
        \draw[->] (R1) -- (R21);
        \draw[->] (R12) -- (R212);
        
        \draw[->, dashed] (8.33,0.55) -- (8.33,-0.55);
        \draw[->, dashed] (8.47,-0.55) -- (8.47,0.55);
        \draw[->, dashed] (R21) -- (R12);
        
        \node[draw=mypink, dashed, fit=(R212), inner sep=2pt] {};
        \node[draw=mypink, dashed, fit=(R12), inner sep=2pt] {};
        
        \node at (9.8,-1.85) {\(0 \xrightarrow[]{}\boldsymbol{\substack{2\\1}} \xrightarrow[]{} \boldsymbol{\substack{1\\2\\1}} \xrightarrow[]{} \boldsymbol{\substack{2\\1}}\)};
        
        \end{tikzpicture}
        \end{center}
        
        The sequences displayed below the three quivers are the images under \(L\), \(C\), and \(R\), respectively, of the Auslander--Reiten sequence in \(\mod{A}\). Since \(\gdim{A}=\infty\), the table above shows that none of these three functors is exact. The picture records the different failures of exactness: for \(L\) the failure occurs at the left-hand term, for \(R\) it occurs at the right-hand term, and for \(C\) it occurs at the middle term.
        \end{ex}

        \paragraph{\bf Acknowledgement.} We thank Peter J{\o}rgensen and Marvin Plogmann for useful discussions. We particularly thank Peter for introducing us to Example~\ref{E: Auslander algebras2}. The second author was supported by the Researcher Project for Early Career Scientists (FRIPRO) from the Norwegian Research Council (project no. 356106).

\bibliographystyle{alpha}
\bibliography{references}

\begin{flushleft}
\textsc{Department of Mathematical Sciences, NTNU}\\
\textsc{NO-7491 Trondheim, Norway}\\
Email address: \href{mailto:david.nkansah@ntnu.no}{\nolinkurl{david.nkansah@ntnu.no}}
\end{flushleft}

\end{document}